\documentclass[10pt, reqno]{amsart}

\usepackage{pgfplots}
\pgfplotsset{compat=1.18}
\usepackage{hyperref}
\usepackage{amssymb, amsmath, amsthm, amsfonts,mathrsfs}
\usepackage{verbatim}
\usepackage{bbm}
\usepackage{enumerate}
\usepackage{dsfont}
\usepackage{upgreek}
\usepackage[mathscr]{eucal}
\usepackage{tikz}
\usetikzlibrary{arrows.meta, positioning}
\usepackage[normalem]{ulem}
\usepackage{epsfig}
\usepackage{pst-grad} 
\usepackage{pst-plot} 
\usepackage{geometry}
\usepackage[backgroundcolor=black!5,linecolor=black]{todonotes}
\usepackage{comment}
\usepackage{booktabs} 
\usepackage{tabularx}
\usepackage{caption} 
\usepackage{tablefootnote}
\usepackage{multirow}
\usepackage{enumerate}
\usepackage{float}

\newtheorem{theorem}{Theorem}[section]
\newtheorem{definition}[theorem]{Definition}
\newtheorem{proposition}[theorem]{Proposition}
\newtheorem{corollary}[theorem]{Corollary}
\newtheorem{lemma}[theorem]{Lemma}

\numberwithin{equation}{section}
\newtheorem{remark}{Remark}

\begin{document}
	\title[Dispersion for  Landau Hamiltonian on the Hyperbolic Plane ]{Long-Time Dispersion and Threshold Resonances for the Landau Hamiltonian on the Hyperbolic Plane}

	\author[H. Guo]{Huanqing Guo}
	\address{Huanqing Guo
		\newline \indent  Institute of Applied Physics and Computational Mathematics, Beijing, 100088, China}
	\email{guohuanqing23@gscaep.ac.cn}
	\author[H. Wang]{Haoran Wang}
	\address{Haoran Wang
	\newline \indent Institute of Mathematics, Henan Academy of Sciences, Zhengzhou 450046, China}
	\email{wanghaoran@hnas.ac.cn}
	
\author[J. Zhang]{Junyong Zhang}
\address{Junyong Zhang\newline
	Department of Mathematics, Beijing Institute of Technology,
	Beijing 100081, P.\,R.\ China}
\email{zhang\_junyong@bit.edu.cn}

	\author[J. Zheng]{Jiqiang Zheng}
	\address{Jiqiang Zheng
		\newline Institute of Applied Physics and Computational Mathematics, Beijing, 100088, China.
		\newline National Key Laboratory of Computational Physics, Beijing, 100088, China}
	\email{zheng\_jiqiang@iapcm.ac.cn, zhengjiqiang@gmail.com}	
	
	\begin{abstract}
		A uniform magnetic field completely suppresses long-time dispersion for the Euclidean Landau Hamiltonian. We show that negative curvature fundamentally changes this behavior. For the Landau Hamiltonian with arbitrary field strength $B\in\mathbb R$ on the hyperbolic plane $\mathbb H$, we establish sharp dispersive estimates for the Schr\"odinger evolution on the absolutely continuous spectral subspace. The decay is $|t|^{-1}$ at short times and, strikingly, improves to $|t|^{-3/2}$ at large times whenever $B\notin\frac12+\mathbb Z$. At half-integer fields, we identify a threshold resonance and isolate its leading contribution, which yields the slower but sharp rate $|t|^{-1/2}$. Thus the long-time dynamics exhibit a resonance-driven dichotomy with no analogue in the flat constant-field problem. Our analysis is based on an explicit representation of the magnetic Schr\"odinger kernel together with a precise separation of the Landau levels and threshold profile. As a further consequence, for every nonthreshold field we obtain global-in-time Strichartz estimates on the absolutely continuous spectral subspace throughout the full hyperbolic admissible region.
	\end{abstract}
	\maketitle
	
		\begin{center}
		\begin{minipage}{120mm}
		\end{minipage}
	\end{center}
	
	\section{Introduction}
	\label{sec:introduction}
	In the symmetric gauge, the classical Landau Hamiltonian $\mathcal{L}_B$ on \(\mathbb{R}^2\) is~\cite{Lan30}
	\begin{equation}\label{Lan:euclid}
		\mathcal{L}_B:=(i\nabla+A(x))^2,\quad \nabla=\left(\frac{\partial}{\partial x_1},\frac{\partial}{\partial x_2}\right),\quad A(x)=\frac{B}{2}(-x_2,x_1),
	\end{equation}
	which models the motion of a charged particle in a constant perpendicular magnetic field of strength $B$. For $B\ne0$, its spectrum is pure point,
	\[
	 \sigma(\mathcal L_B)=\{(2n+1)|B|:n\in\mathbb N_0\},
	\]
	and every eigenvalue has infinite multiplicity; these eigenvalues are referred to as the (Euclidean) Landau levels. The same spectral confinement is visible directly in the Schr\"odinger propagator. The Mehler formula for the twisted Laplacian gives, whenever $t\notin \pi |B|^{-1}\mathbb Z$,
	\begin{equation}\label{eq:euclidean propagator norm}
	 \bigl\|e^{-it\mathcal L_B}\bigr\|_{L^1(\mathbb R^2)\to L^\infty(\mathbb R^2)}
	 =\frac{|B|}{4\pi|\sin(Bt)|};
	\end{equation}
	see, for example,~\cite[Chapter~1]{Thangavelu1993} and~\cite[Chapter~1]{Folland1989}. Thus, in $\mathbb{R}^2$, the magnetic Schr\"odinger flow $e^{-it\mathcal L_B}$ behaves (in the $L^1\to L^\infty$ norm) like $|t|^{-1}$ as $t\to0$, but it is periodic up to a phase and refocuses at the times $t\in\pi |B|^{-1}\mathbb Z$. In particular, it admits no dispersive bound that decays as $|t|\to\infty$. This failure of long-time dispersion in the flat, constant-field problem is the principal motivation for asking whether negative curvature changes the picture.

    We study the Landau Hamiltonian $H_B$ on the hyperbolic plane $\mathbb{H}$, which also arises naturally in the two-dimensional quantum Hall effect; see, for example,~\cite{CHMM98}. The upper half-plane
	\[
	\mathbb{H} = \{z = x+iy \in \mathbb{C} : y > 0\}
	\]
	equipped with the Poincar\'{e} metric \(g:=y^{-2}(dx^2 + dy^2)\) forms a complete, simply connected, two-dimensional Riemannian manifold of constant negative curvature \(-1\).  
	The uniform magnetic field on \(\mathbb{H}\) is represented by the curvature form \(F_B:= iBy^{-2}dx\wedge dy\), where \(B\in\mathbb{R}\) is the field strength. Any real $1$-form \(A\) satisfying \(i\,dA = F_B\) is a vector potential for \(F_B\). We work in the gauge
	\begin{equation}
		\label{eq:gauge}
		A = B y^{-1}dx.
	\end{equation}
	Let \(\mathbf{L}_B=\mathbb{H}\times \mathbb{C}\) be the trivial Hermitian line bundle with covariant derivative $\nabla^A=d-iA$. Its square-integrable sections form $L^2(\mathbb H,d\omega)$, where $d\omega(z)=y^{-2}dx\,dy$, with inner product
	\[(f_1,f_2):=\int_{\mathbb{H}} \langle f_1,f_2\rangle_z\,d\omega(z).\]
    The Landau Hamiltonian \(H_B\) is the associated connection Laplacian,
	\begin{equation}
		\label{eq:LandauHamiltonian}
		H_B = -y^2\Bigl(\partial_x - i\frac{B}{y}\Bigr)^2 - y^2\partial_y^2.
	\end{equation}	
	By~\cite{Shubin2001}, the operator initially defined on $C_c^\infty(\mathbb H)$ is positive and essentially self-adjoint in \(L^2(\mathbb{H},d\omega)\). We use the same symbol \(H_B\) for its unique self-adjoint extension.

	We investigate the dispersive properties of the corresponding Schr\"odinger equation,
	\begin{equation}
		\label{eq:prob}
		i\partial_t u = H_B u,\qquad u=u(t,z):\mathbb{H}\times\mathbb{R}\to\mathbb{C},
	\end{equation}
	which describes the nonrelativistic dynamics of a quantum particle in the fixed external magnetic field \(F_B=i\,dA\). Negative curvature changes the flat spectral picture decisively. Comtet~\cite{Comtet1987} proved, by group-theoretic methods, that the spectrum of \(H_B\) in \(L^2(\mathbb{H},d\omega)\) consists of the absolutely continuous interval
	\(	\bigl[\frac{1}{4}+B^2,+\infty\bigr)	\) and finitely many eigenvalues, the hyperbolic Landau levels,
	\[
	E_n(B):= \frac{1}{4}+B^2-\bigl(n+\tfrac12-|B|\bigr)^2,\qquad 0\le n<|B|-\tfrac12,
	\]
	each of infinite multiplicity. The point spectrum is therefore empty when \(|B|\leqslant 1/2\). More generally, after the finitely many Landau eigenvalues are removed, the continuous spectral component can disperse. The free hyperbolic flow already displays curvature-enhanced large-time decay~\cite{anker2009,Banica2007}; the problem here is to determine how a uniform magnetic field modifies that mechanism.

	For arbitrary field strength, remove the Landau eigenvalues and write
	\[
	P_{ac}=I-\sum_{0\le n<|B|-\frac12}\Pi_n,
	\]
	where $\Pi_n$ is the spectral projection onto $E_n(B)$. The sum is empty when $|B|\le\frac12$. Our main result shows that negative curvature restores genuine long-time dispersion on this absolutely continuous subspace. It also identifies a sharp threshold dichotomy: the decay is $|t|^{-3/2}$ away from half-integer fields, but a threshold resonance slows it to $|t|^{-1/2}$ at half-integer fields. Both behaviors are impossible for the recurrent Euclidean flow in~\eqref{eq:euclidean propagator norm}.

	\begin{theorem}
		\label{cor:decay}
		Let \(H_B\) be the self-adjoint Hamiltonian \eqref{eq:LandauHamiltonian} with \(B\in\mathbb R\). Then we have the following dispersive bound
		\begin{equation}\label{dis:HB}
			\|e^{-itH_B}P_{ac}\|_{L^1(\mathbb{H})\rightarrow L^\infty(\mathbb{H})}\le C_B\times
            \begin{cases}
                |t|^{-1},&0<|t|\leqslant 1,\\
                |t|^{-\frac{3}{2}},&|t|\geqslant 1,\quad B\notin\frac12+\mathbb Z,\\
                |t|^{-\frac{1}{2}},&|t|\geqslant 1,\quad B\in\frac12+\mathbb Z,
            \end{cases}
		\end{equation}
		where \(C_B>0\) is a constant independent of \(t\).
	\end{theorem}

		\begin{remark}	The short-time rate $|t|^{-1}$ agrees with the local Euclidean rate. The $|t|^{-3/2}$ large-time decay for nonthreshold fields is a curvature effect and has no flat constant-field analogue, since the Euclidean propagator is recurrent. The same improvement from small to large times is known for the free hyperbolic flow~\cite{anker2009,Banica2007}.\end{remark}
			
		\begin{remark}	The distinction between $B\in\frac12+\mathbb Z$ and $B\notin\frac12+\mathbb Z$ is essential. At a half-integer field, the threshold resonance described in Section~\ref{subsec:resonances} survives the projection $P_{ac}$ because it is not an $L^2$ eigenstate. It produces the leading $|t|^{-1/2}$ term rather than eliminating dispersion altogether.
	\end{remark}

	\begin{remark}\label{rem:dispersive} Interpolating between the dispersive estimate \eqref{dis:HB} and the $L^2$ conservation 
			\[\|e^{-itH_B}P_{ac}f\|_{L^2(\mathbb{H})}=\|P_{ac}f\|_{L^2(\mathbb{H})}\leqslant\|f\|_{L^2(\mathbb{H})},\quad\forall t\in\mathbb{R},\]
			we obtain 
            \begin{equation}\label{dis:HB interp}
			\|e^{-itH_B}P_{ac}f\|_{L^p(\mathbb{H})}\le C_B\|f\|_{L^{p'}(\mathbb{H})}\times
            \begin{cases}
                |t|^{-(1-\frac{2}{p})},&0<|t|\leqslant 1,\\
                |t|^{-\frac{3}{2}(1-\frac{2}{p})},&|t|\geqslant 1,\quad B\notin\frac12+\mathbb Z,\\
                |t|^{-\frac{1}{2}(1-\frac{2}{p})},&|t|\geqslant 1,\quad B\in\frac12+\mathbb Z,
            \end{cases}
		\end{equation}
        for all $2\le p\le\infty$.
	\end{remark}
	
	For $B\notin\frac12+\mathbb Z$, the stronger dispersive estimate in Corollary~\ref{cor:decay stronger} yields the corresponding Strichartz estimates through the standard \(TT^*\) argument of Keel--Tao~\cite{keel-tao}. The integrable large-time decay enlarges the range beyond the Euclidean scaling line. Specifically, the admissible pairs \((p,q)\) form the upper triangle
	\[
	\Delta=\Bigl\{ \bigl(\tfrac1p,\tfrac1q\bigr) \in \bigl(0,\tfrac12\bigr]\times\bigl(0,\tfrac12\bigr) : \tfrac2p+\tfrac2q \geqslant 1 \Bigr\}\cup\bigl\{(0,\tfrac12)\bigr\}.
	\]
	
	\begin{corollary}[Strichartz estimate]
		\label{cor:Strichartz}
		Let \(H_B\) be the self-adjoint Hamiltonian \eqref{eq:LandauHamiltonian} with \(B\notin\frac12+\mathbb Z\) and let $u$ be the projected Duhamel solution
		\[
		u(t)=e^{-itH_B}P_{ac}f-i\int_0^t e^{-i(t-s)H_B}P_{ac}F(s)\,ds.
		\]
		Thus $u(0)=P_{ac}f$ and, for smooth data, $i\partial_tu=H_Bu+P_{ac}F$. Then there exists a constant \(C_B>0\) such that the following Strichartz estimate
		\begin{equation}
			\|u\|_{L^p_t(\mathbb{R};L^q_z(\mathbb{H}))}\leqslant C_B\bigl(\|f\|_{L^2_z(\mathbb{H})}+\|F\|_{L^{\tilde{p}'}_t(\mathbb{R};L^{\tilde{q}'}_z(\mathbb{H}))}\bigr)
		\end{equation}
        holds whenever \((1/p,1/q),(1/\tilde{p},1/\tilde{q})\in\Delta\), with \(C_B\) allowed to depend on these pairs. For general $F$ in the indicated space, the Duhamel term is defined by continuous extension from smooth compactly supported forcing.
	\end{corollary}

	The rest of this paper is structured as follows. Section~\ref{sec:preliminaries} collects the geometric, convolution, and special-function preliminaries. Section~\ref{sec:spectral analysis} describes the spectral resolution and resolvent resonances of $H_B$. In Section~\ref{sec:heat-propagator}, we derive the heat and Schr\"odinger kernel representations and establish the corresponding kernel estimates at nonthreshold and threshold fields. Section~\ref{sec:proof of cor} applies these estimates to prove Theorem~\ref{cor:decay} and Corollary~\ref{cor:Strichartz}.
	
	\section{Preliminaries}
	\label{sec:preliminaries}
	
	In this section, we review the geometry and convolution estimates on the hyperbolic plane, and collect the hypergeometric identities used in Sections~\ref{sec:spectral analysis} and~\ref{sec:heat-propagator}.
	
	In what follows, we shall employ the notation \(A = O(B)\) or \(A \lesssim B\) to denote \(A \le C B\) for some constant \(C > 0\). If both \(A \lesssim B\) and \(B \lesssim A\) hold, we simply write \(A \sim B\).
	
\subsection{Geometry and convolution estimates}

There are several equivalent models of the hyperbolic plane. One is the upper half-plane in \(\mathbb{C}\),
\[
\mathbb{H} := \{ z = x+iy \in \mathbb{C} : y > 0 \}
\]
equipped with the hyperbolic metric $ds^2 = y^{-2}(dx^2 + dy^2)$. The hyperbolic distance $r=r(z,z')$ between two points \(z = x+iy,z' = x'+iy' \in \mathbb{H}\) is defined via
\[
\cosh r = 1 + \frac{|z-z'|^2}{2\operatorname{Im} z\operatorname{Im} z'}.
\]
For later use, we introduce 
\begin{equation}
	\label{eq:hyperbolic distance1}
	\sigma(z,z') := \frac{\cosh r+1}{2} = \frac{|z-\bar{z'}|^2}{4yy'}.
\end{equation}
The hyperbolic plane \(\mathbb{H}\) can be alternatively realized as the upper sheet of the standard hyperboloid in \(\mathbb{R}^3\) 
	\[
	\{(x_0,x_1,x_2):x_0^2-x_1^2-x_2^2=1,\;x_0\geqslant 1\}
	\]
	 equipped with the Minkowski metric
	\[
	dl^2=-dx_0^2+dx_1^2+dx_2^2,
	\]
	or as the homogeneous space \(\mathbb{G}/\mathbb{K}\) with \(\mathbb{G}=\operatorname{SL}_2(\mathbb{R})\) and \(\mathbb{K}=SO(2)\).
	
	The group \(\mathbb{G}=\operatorname{SL}_2(\mathbb{R})\) is semisimple and therefore unimodular, while \(\mathbb{K}=SO(2)\) is compact. We normalize the Haar measures on \(\mathbb{K}\) and \(\mathbb{G}\) such that \(\int_{\mathbb{K}}1\,dk=1\) and
	\[
	\int_{\mathbb{G}}f(g\cdot \mathbf{0})\,dg = \int_{\mathbb{H}} f(x)\,d\omega
	\]
	for every \(f\in C_c(\mathbb{H})\). Given two functions \(f_1,f_2\in C_c(\mathbb{G})\), their convolution is defined by
	\[
	(f_1*f_2)(h)=\int_{\mathbb{G}}f_1(g)f_2(g^{-1}h)\,dg.
	\]
	A function \(f:\mathbb{G}\to \mathbb{C}\) is called \(\mathbb{K}\)-biinvariant if
	\[
	f(k_1gk_2)=f(g)\qquad \forall k_1,k_2\in \mathbb{K}.
	\]
	Likewise, a function \(f:\mathbb{H}\to \mathbb{C}\) is called \(\mathbb{K}\)-invariant (or radial) if
	\[
	f(k\cdot z)=f(z)\qquad \forall k\in \mathbb{K},\;z\in \mathbb{H}.
	\]
	If \(f,K\in C_c(\mathbb{H})\) and \(K\) is \(\mathbb{K}\)-invariant, we set
	\[
	(f*K)(z)=\int_{\mathbb{G}}f(g\cdot \mathbf{0})\,K(g^{-1}\cdot z)\,dg.
	\]
	
	We now recall the celebrated Kunze–Stein phenomenon. A sharp version, due to Cowling, Meda and Setti (see \cite{cowling1997}) and subsequently improved by Ionescu \cite{ionescu2000}, states as follows.
	
	\begin{proposition}
		\label{thm:the Kunze–Stein phenomenon}
		If \(G\) is a noncompact connected semisimple Lie group of real rank one with finite center, then one has \footnote{The usual convention is that if \(\mathcal{U},\mathcal{V}\) and \(\mathcal{W}\) are Banach spaces of functions on \(G\), the notation \(\mathcal{U}*\mathcal{V}\subset \mathcal{W}\) indicates both the set inclusion and the associated norm inequality.}
		\[
		L^{2,1}(G)*L^{2,1}(G)\subset L^{2,\infty}(G),
		\]
		which implies 
		\[
		L^{p,u}(G)*L^{p,v}(G)\subset L^{p,w}(G)
		\]
		for \(p\in (1,2)\) and \((u,v,w)\in [1,\infty]^3\) satisfying \(1+1/w < 1/u+1/v\). Here the Lorentz spaces \(L^{q,\alpha}(G)\) are variants of the classical Lebesgue spaces with norm defined by
		\[
		\| f \|_{L^{q,\alpha}} = \begin{cases} 
			\displaystyle\Bigl[ \int_0^{+\infty} \bigl( s^{1/q} f^*(s) \bigr)^\alpha \,\frac{ds}{s} \Bigr]^{1/\alpha}, & 1 \leqslant \alpha < \infty,\\[1.2em]
			\displaystyle\sup_{s > 0} s^{1/q} f^*(s), & \alpha = \infty,
		\end{cases}
		\]
		where \(f^*\) denotes the decreasing rearrangement of \(f\).
	\end{proposition}
	
	From Proposition~\ref{thm:the Kunze–Stein phenomenon}, one immediately deduces
	\[
	L^{q'}(\mathbb{K}\setminus G) * L^{q'}(G/\mathbb{K}) \subset L^{q',\infty}(\mathbb{K}\setminus G/\mathbb{K}) \qquad \forall q > 2.
	\]
	By this inclusion we mean that there exists a constant \(C_q > 0\) such that
	\[
	\| f * g \|_{L^{q',\infty}} \leqslant C_q \| f \|_{L^{q'}} \| g \|_{L^{q'}} \qquad \forall f \in L^{q'}(\mathbb{K}\setminus G),\;\forall g \in L^{q'}(G/\mathbb{K}).
	\]
	Hence, by duality, we have
	\begin{equation}
		\label{eq:the Kunze–Stein phenomenon1}
		\|f*g\|_{L^q} \leqslant C_q \|f\|_{L^{q'}}\,\|g\|_{L^{q,1}} \qquad f\in L^{q'}(G/\mathbb{K}),\;\forall g\in L^{q,1}(\mathbb{K}\setminus G/\mathbb{K}).
	\end{equation}
	
	These group‑level estimates can be transferred to the symmetric space by the natural identifications between functions on \(\mathbb{H}\) and \(\mathbb{K}\)-invariant functions on \(G = \operatorname{SL}_2(\mathbb{R})\).
	
	For any \(f \in C_c(\mathbb{H})\) define \(\tilde f(g) = f(g\cdot \mathbf{0})\). Then \(\tilde f \in C_c(G/\mathbb{K})\) and \(\|f\|_{L^{p,q}(\mathbb{H})} = \|\tilde f\|_{L^{p,q}(G/\mathbb{K})}\).  
	A function \(K \in C_c(\mathbb{H})\) is radial (i.e., \(K(k\cdot x) = K(x)\)) if and only if \(\tilde K\) is \(\mathbb{K}\)-biinvariant. Moreover,
	\[
	(f * K)(x) = \int_G f(g\cdot \mathbf{0})\,K(g^{-1}\cdot x)\,dg
	\]
	translates into group convolution as \(\widetilde{(f * K)}(h) = (\tilde f * \tilde K)(h)\). The dual inequality \eqref{eq:the Kunze–Stein phenomenon1} therefore yields the corresponding estimate on \(\mathbb{H}\):
	\begin{equation}
		\label{eq:the Kunze–Stein phenomenon2}
		\|f * K\|_{L^q(\mathbb{H})} \le C_q \|f\|_{L^{q'}(\mathbb{H})} \|K\|_{L^{q,1}_{\mathrm{rad}}(\mathbb{H})} \qquad \forall f \in L^{q'}(\mathbb{H}),\; K \in L^{q,1}_{\mathrm{rad}}(\mathbb{H}),
	\end{equation}
	where \(L^{q,1}_{\mathrm{rad}}(\mathbb{H})\) denotes the subspace of radial Lorentz functions.

\subsection{Hypergeometric functions and Jacobi polynomials}
\label{subsec:hypergeometric}

Throughout this subsection, $\mathbb N_0=\{0,1,2,\ldots\}$, and the Pochhammer symbol is defined by $(a)_0=1$ and $(a)_k=\prod_{j=0}^{k-1}(a+j)$ for $k\ge1$. The Gauss hypergeometric function is given by
\begin{equation}\label{eq:hypergeom-series}
 F(a,b;c;\zeta):={}_2F_1(a,b;c;\zeta)
 =\sum_{k=0}^{\infty}\frac{(a)_k(b)_k}{(c)_k\,k!}\zeta^k,
 \qquad |\zeta|<1,\quad c\notin-\mathbb N_0,
\end{equation}
and extended to $\mathbb C\setminus[1,\infty)$ on its principal branch; see \cite[\S15.2]{DLMF}. It is symmetric in $a,b$, and the defining series terminates as soon as one of the numerator parameters is a nonpositive integer. Exceptional values of the parameters are interpreted throughout by meromorphic continuation of the complete expressions, with removable singularities filled in.

We shall use the Pfaff transformation \cite[Eq.~(15.8.1)]{DLMF}
\begin{equation}\label{eq:hypergeom-pfaff}
 F(a,b;c;\zeta)
 =(1-\zeta)^{-a}F\left(a,c-b;c;\frac{\zeta}{\zeta-1}\right),
 \qquad |\arg(1-\zeta)|<\pi,
\end{equation}
which is valid for $c\notin-\mathbb N_0$, the power being taken on the principal branch. We shall also use the Chu--Vandermonde identity \cite[Eq.~(15.4.24)]{DLMF}
\begin{equation}\label{eq:hypergeom-vandermonde}
 F(-n,b;c;1)=\frac{(c-b)_n}{(c)_n},
 \qquad n\in\mathbb N_0,\quad c\notin-\mathbb N_0.
\end{equation}

The resolvent continuation requires the zero-balanced expansion, in which the denominator parameter is the sum of the two numerator parameters. Writing $\psi=\Gamma'/\Gamma$, we obtain from \cite[Eq.~(15.8.10), with $m=0$]{DLMF}
\begin{equation}\label{eq:hypergeom-zero-balanced}
 \begin{split}
F(a,b;a+b;1-u)
 ={}& \frac{\Gamma(a+b)}{\Gamma(a)\Gamma(b)}\sum_{k=0}^{\infty}\frac{(a)_k(b)_k}{(k!)^2}u^k\\
 &\quad\times\bigl[2\psi(k+1)-\psi(a+k)-\psi(b+k)-\log u\bigr],
 \end{split}
\end{equation}
first for $0<u<1$ and $a,b,a+b\notin-\mathbb N_0$, and then, by meromorphic continuation, in $a,b$. Recall that the digamma function $\psi$ has simple poles at $-\mathbb N_0$, each with residue $-1$ \cite[\S5.2]{DLMF}. Once the $\log u$ term has been separated off and the finitely many possible parameter poles have been cleared locally, the two resulting ordinary power series in $u$, together with their parameter derivatives of any fixed order, converge normally for $0\le u\le u_0<1$, locally uniformly in the parameters. This assertion pertains to the two power series alone; the complete expression in \eqref{eq:hypergeom-zero-balanced} retains its logarithmic singularity at $u=0$.

We shall only need Jacobi polynomials whose first parameter vanishes. Their terminating hypergeometric representation is \cite[Eq.~(18.5.7)]{DLMF}
\begin{equation}\label{eq:jacobi-hypergeom}
 P_n^{(0,\beta)}(x)
 =F\left(-n,n+\beta+1;1;\frac{1-x}{2}\right),
 \qquad n\in\mathbb N_0.
\end{equation}
For every $\beta\in\mathbb C$ this defines a polynomial of degree at most $n$, in particular for the negative parameters $\beta=-2|B|$ used below; no assumption that $\beta$ lies in the range of orthogonality is required. We shall also use the finite binomial expansion \cite[Eq.~(18.5.8)]{DLMF}
\begin{equation}\label{eq:jacobi-binomial}
 P_n^{(0,\beta)}(x)
 =\sum_{j=0}^{n}\binom nj\binom{n+\beta}{j}
 \left(\frac{x-1}{2}\right)^j
 \left(\frac{x+1}{2}\right)^{n-j}.
\end{equation}
Here the generalized binomial coefficients are defined by
\[
 \binom zj=\frac{z(z-1)\cdots(z-j+1)}{j!}
 =\frac{(-1)^j(-z)_j}{j!},\qquad j\ge1,
 \qquad\binom z0=1.
\]
Consequently, \eqref{eq:jacobi-binomial} is a polynomial identity valid for all $x,\beta\in\mathbb C$, including negative integral $\beta$.

The generating function \cite[Eq.~(18.12.1)]{DLMF} then reads
\begin{equation}\label{eq:jacobi-generating}
 \sum_{n=0}^{\infty}P_n^{(0,\beta)}(x)q^n
 =\frac{1}{\sqrt{1-2xq+q^2}}
 \left(\frac{1+q+\sqrt{1-2xq+q^2}}{2}\right)^{-\beta}.
\end{equation}
For $x\ge1$ the identity holds when $|q|<x-\sqrt{x^2-1}$, with the square root normalized to equal $1$ at $q=0$ and the remaining power defined by the analytic logarithm that vanishes there. By analytic continuation, or equivalently by comparing its Taylor coefficients, the identity extends to every $\beta\in\mathbb C$. In particular, for $x=\cosh r$ with $r\ge0$, the indicated range of convergence is $|q|<e^{-r}$.

Finally, the Legendre polynomials correspond to the case $\beta=0$:
\begin{equation}\label{eq:legendre-hypergeom-bound}
 \begin{gathered}
 P_n(x)=P_n^{(0,0)}(x)
 =F\left(-n,n+1;1;\frac{1-x}{2}\right),\\
 |P_n(x)|\le1\qquad(-1\le x\le1).
 \end{gathered}
\end{equation}
This bound is the specialization of \cite[Eq.~(18.14.1)]{DLMF} to vanishing parameters, and it will be used for the threshold radial profile.

\section{Spectral theory of the Landau Hamiltonian}
\label{sec:spectral analysis}
We record the resolvent and spectral resolution needed for the kernel estimates, following \cite{Fay77,Comtet1987,LudewigThiang21}, and then analyze the continued resolvent. Throughout Sections~\ref{sec:spectral analysis}--\ref{sec:proof of cor}, $B\in\mathbb R$.

\subsection{Resolvent kernel and spectral resolution}
We consider the Maass Laplacian
\begin{equation}\label{eq:Maass laplacian}
 D_B:=B^2-H_B=y^2(\partial_x^2+\partial_y^2)-2iBy\partial_x.
\end{equation}
For $z,z'\in\mathbb H$, write $r=d(z,z')$, and fix the magnetic phase
\begin{equation}\label{eq:magnetic phase}
 \Phi_B(z,z'):=\left(\frac{z'-\bar z}{z-\bar z'}\right)^B
 =\exp\!\left(iB[\pi-2\operatorname{Arg}(z-\bar z')]\right),
 \qquad \operatorname{Arg}(z-\bar z')\in(0,\pi).
\end{equation}
Thus $|\Phi_B|=1$ and $\Phi_B(z,z)=1$. Only the following two radial functions will be used in our proof. In the notation of \cite[Eq.~(13)]{Fay77}, they are
\begin{equation}
	\label{eq:P,Q,snB}
\begin{split}
 P^0_{s,B}(r)&=\sigma^{-s}F(s-|B|,s+|B|;1;1-\sigma^{-1}),\\
 Q^0_{s,B}(r)&=-\frac{\Gamma(s-|B|)\Gamma(s+|B|)}{4\pi\Gamma(2s)}
 \sigma^{-s}F(s+|B|,s-|B|;2s;\sigma^{-1}),
\end{split}
\end{equation}
where $\sigma=\cosh^2(r/2)$ is defined in \eqref{eq:hyperbolic distance1} and $F$ is the Gauss hypergeometric function with the conventions of Section~\ref{subsec:hypergeometric}. These radial factors are even in $B$, and the connection formulas \cite[Eq.~(20)]{Fay77} give
\begin{equation}\label{eq:Q-symmetry}
 P^0_{s,B}=P^0_{1-s,B},\qquad
 Q^0_{s,B}-Q^0_{1-s,B}
 =-\frac{\sin(2\pi s)}{4\sin\pi(s+B)\sin\pi(s-B)}P^0_{s,B}.
\end{equation}
For fixed $s$ away from the poles, the normalizations at the diagonal are
\begin{equation}\label{eq:asymptotic properties of P,Q}
\begin{split}
    P^0_{s,B}(r)&=1+O(r^2),\\
    Q^0_{s,B}(r)&=\frac1{2\pi}\log r+O(1)
 \quad(r\downarrow0).
\end{split}
\end{equation}
Also $Q^0_{s,B}(r)=O_{s,B}(e^{-\operatorname{Re}(s)r})$ as $r\to\infty$ in the physical half-plane $\operatorname{Re}s>\frac12$, away from the poles. And for $\operatorname{Re}s>\frac12$ and $B^2+s(1-s)\notin\sigma(H_B)$, the kernel of $(D_B+s(1-s))^{-1}$ is
\begin{equation}\label{eq:green function}
 g_{s,B}(z,z')=\Phi_B(z,z')Q^0_{s,B}(r);
\end{equation}
see \cite[proof of Theorem~1.4, p.~152]{Fay77} and \cite[Section~V]{Comtet1987}. In particular, the kernel of $(H_B-B^2-s(1-s))^{-1}$ is $-g_{s,B}$.

The spectrum consists of the absolutely continuous part $[E_\ast(B),\infty)$, where $E_\ast(B)=B^2+\frac14$, and the Landau eigenvalues. For use here and in Section~\ref{sec:heat-propagator}, set
\begin{equation}\label{eq:landau parameters}
 \begin{split}
 N&=\#\{n\in\mathbb N_0:n<|B|-\tfrac12\},\\
 s_n&=|B|-n,\qquad 0\le n<N.
 \end{split}
\end{equation}
Then
\[
 E_n(B)=E_\ast(B)-(s_n-\tfrac12)^2,\qquad 0\le n<N,
\]
with all Landau sums understood to be empty when $N=0$.
There is no singular continuous spectrum or threshold eigenvalue; see \cite[Section~III]{Comtet1987} and the spectral resolution below. Each Landau eigenspace has infinite dimension. Negative fields follow by complex conjugation, which intertwines $H_B$ and $H_{-B}$.

The connection formula \eqref{eq:Q-symmetry} and Stone's formula give the following form of Fay's spectral resolution, with the eigenfunction sum grouped into Landau projections.

\begin{lemma}[Spectral resolution; {\cite[Theorem~1.4, Eqs.~(29)--(31)]{Fay77}}]\label{lem:fay}
	For $B\in\mathbb R$ and $f\in C_c^\infty(\mathbb H)$,
	\begin{equation}\label{eq:spectral decomposition}
	\begin{split}
	 f(z)={}&\sum_{n=0}^{N-1}(\Pi_nf)(z)
	 +\frac1{8\pi i}\int_{\operatorname{Re}s=1/2}
	 \frac{(s-\frac12)\sin(2\pi s)}{\sin\pi(s-B)\sin\pi(s+B)}\\
	 &\hspace{12mm}\times\int_{\mathbb H}
	 \Phi_B(z,z')P^0_{s,B}(d(z,z'))f(z')\,d\omega(z')\,ds.
	\end{split}
	\end{equation}
	The contour is oriented upwards.
\end{lemma}

Writing $s=\tfrac12+i\nu$, the contour term gives
\[
 (P_{ac}f)(z)
 =\int_0^\infty\frac{\nu\sinh(2\pi\nu)}{2\pi[\cosh(2\pi\nu)+\cos(2\pi B)]}
 \int_{\mathbb H}\Phi_B(z,z')P^0_{1/2+i\nu,B}(d(z,z'))f(z')\,d\omega(z')\,d\nu.
\]
The points $s_n=|B|-n>\frac12$ represent the Landau eigenvalues, whereas the line $\operatorname{Re}s=\frac12$ parametrizes the continuous spectrum. Figure~\ref{fig:spectrum} also marks the continued poles classified in the next subsection.

\begin{figure}[H]
	\centering
	\begin{tikzpicture}[>=Stealth]
		
		\fill[green!6] (0.5,-4.5) rectangle (4.5,4.5);
		\node[font=\scriptsize\itshape, opacity=0.6]
		at (3.5,-4.0) {Physical sheet};
		
		\draw[->] (-3.2,0) -- (5.0,0)
		node[right, font=\normalsize]{$\operatorname{Re}(s)$};
		\draw[->] (0,-4.5) -- (0,4.8)
		node[above, font=\normalsize]{$\operatorname{Im}(s)=\nu$};
		
		\foreach \x in {-3,-2,-1,1,2,3,4}{
			\draw[gray!40] (\x,0.07) -- (\x,-0.07);
			\node[below, font=\tiny, gray!55] at (\x,-0.1) {$\x$};
		}
		\foreach \y in {-4,-3,-2,-1,1,2,3,4}{
			\draw[gray!40] (0.07,\y) -- (-0.07,\y);
			\node[left, font=\tiny, gray!55] at (-0.1,\y) {$\y$};
		}
		
		\draw[blue!75!black,dashed] (0.5,-4.5) -- (0.5,4.5);
		\node[blue!75!black, font=\scriptsize\itshape] at (0.5,-5.0) {$s=\frac{1}{2}$};
		%
		\node[circle, fill=blue, inner sep=1.5pt, label=above:{$s_0$}] at (3,0) {};
		\node[circle, fill=blue, inner sep=1.5pt, label=above:{$s_1$}] at (2,0) {};
		\node[circle, fill=blue, inner sep=1.5pt, label=above:{$s_2$}] at (1,0) {};
        \foreach \x in {0,-1,-2,-3}{
            \node[circle, draw=blue, fill=white, line width=0.8pt,
            inner sep=2.1pt] at (\x,0) {};
        }
        \node[below right, font=\tiny, gray!55] at (0,-0.1) {$0$};
        \node[blue, font=\small] at (-1.5,0.55) {Resonances};

	\end{tikzpicture}
	\caption{Eigenvalues and resonances of $H_B$ for $B=3$. Filled circles mark the eigenvalues on the physical sheet. Open circles mark the resonance poles.}
	\label{fig:spectrum}
\end{figure}

\subsection{Resolvent poles and threshold resonances}
\label{subsec:resonances}
Write
\[
R_B(s)=\bigl(H_B-B^2-s(1-s)\bigr)^{-1}.
\]
Initially $\operatorname{Re}s>\frac12$, away from the Landau levels, and the kernel is $-g_{s,B}=-\Phi_BQ^0_{s,B}$. The parameters
\[
E(s)=B^2+s(1-s)=E_\ast(B)-(s-\tfrac12)^2=E_\ast(B)+k^2,
\qquad k=i(s-\tfrac12),
\]
identify $\operatorname{Re}s>\frac12$ with the physical sheet $\operatorname{Im}k>0$ and $\operatorname{Re}s<\frac12$ with the nonphysical sheet; $s=\frac12$ is the threshold (see Figure~\ref{fig:spectrum}).  

We next continue $R_B(s)$ from the physical sheet. To specify the analytic meaning of this continuation, fix any nonzero $\chi\in C_c^\infty(\mathbb H;\mathbb R)$ and set
\begin{equation}\label{eq:cutoff resolvent}
	\mathcal H=L^2(\mathbb H,d\omega),\qquad
	\mathcal B=\mathcal L(\mathcal H),\qquad
	T_\chi(s)=\chi R_B(s)\chi\in\mathcal B,
\end{equation}
initially on the physical sheet away from the eigenvalue poles. For an open set $U\subset\mathbb C$, $\mathcal O(U;\mathcal B)$ denotes the set of all the strongly holomorphic maps $U\to\mathcal B$ in the operator norm. For fixed $\chi$, since multiplication by $\chi$ is bounded on $\mathcal H$, we have
\begin{equation}\label{eq:physical holomorphy}
	T_\chi\in\mathcal O(\Omega_B^+;\mathcal B),\qquad
	\Omega_B^+=\{s\in\mathbb C:\operatorname{Re}s>\tfrac12\}
	\setminus\{s_n:0\le n<N\}.
\end{equation}
The continuation beyond this domain is established next.

\begin{proposition}\label{prop:resonance poles}
	For every $0\ne\chi\in C_c^\infty(\mathbb H;\mathbb R)$, $T_\chi$ has a unique meromorphic continuation from $\Omega_B^+$ to $\mathbb C$ as a $\mathcal B$-valued map. Its poles are precisely
	\begin{equation}\label{eq:resonance poles}
		\mathcal P_B=\{|B|-n:n\in\mathbb N_0\}\cup\{-|B|-n:n\in\mathbb N_0\},
	\end{equation}
	and all are simple. More precisely, for each $s_0\in\mathcal P_B$ and a sufficiently small disk $U$ about $s_0$,
	\begin{equation}\label{eq:cutoff Laurent expansion}
		T_\chi(s)=\frac{\chi A_{s_0}\chi}{s-s_0}+H_{\chi,s_0}(s),\qquad
		H_{\chi,s_0}\in\mathcal O(U;\mathcal B),\quad \chi A_{s_0}\chi\ne0.
	\end{equation}
	With $ \mu_B(s_0)=\mathbf1_{\{s_0-|B|\in-\mathbb N_0\}}
	+\mathbf1_{\{s_0+|B|\in-\mathbb N_0\}}$,
	the residue kernels are
	\begin{equation}\label{eq:resonance residue}
		A_{s_0}(z,z'):=\operatorname*{Res}_{s=s_0}R_B(s;z,z')
		=\frac{\mu_B(s_0)}{4\pi}\Phi_B(z,z')P^0_{s_0,B}(r).
	\end{equation}
\end{proposition}

\begin{proof}
	We recall from \eqref{eq:P,Q,snB}-\eqref{eq:green function} that
	\begin{align*}
		R_B(s;z,z')
		&=-\Phi_B(z,z')Q^0_{s,B}(r)\\
		&=\Phi_B(z,z')\cdot \frac{\Gamma(s-|B|)\Gamma(s+|B|)}{4\pi\Gamma(2s)}
		\sigma^{-s}F(s+|B|,s-|B|;2s;\sigma^{-1}).
	\end{align*}
	For $s\in\Omega_B^+$ and $r>0$, set $u=1-\sigma^{-1}=\tanh^2(r/2)$. Here $s\pm |B|\notin-\mathbb N_0$ and $0<u<1$, so the zero-balanced hypergeometric expansion \eqref{eq:hypergeom-zero-balanced} applies directly to the physical Green kernel \eqref{eq:green function} and gives
	\begin{equation}
		\label{eq:resonance expansion}
		\begin{split}
			\frac{R_B(s;z,z')}{\Phi_B(z,z')}
			={}&\frac{\sigma^{-s}}{4\pi}
			\sum_{k=0}^\infty\frac{(s-|B|)_k(s+|B|)_k}{(k!)^2}u^k\\
			&\quad\times\bigl[2\psi(k+1)-\psi(s-|B|+k)-\psi(s+|B|+k)-\log u\bigr],
		\end{split}
	\end{equation}
	We use the right-hand side to construct the continuation. The Pochhammer factors are polynomials, while $\psi$ has simple poles at the nonpositive integers with residue $-1$. After clearing the finitely many possible poles on a compact $s$-set, the coefficients and their parameter derivatives grow at most polynomially in $k$, up to logarithms. After separating the $\log u$ term, these power series converge for $0\le u\le u_0<1$, locally uniformly in $s$.

	It remains to justify the passage from this pointwise expansion to an operator-valued continuation. Let $K=\operatorname{supp}\chi$. Since $K$ is compact, $r=d(z,z')$ is bounded on $K\times K$, and hence $u=\tanh^2(r/2)\le u_0<1$. After multiplication by a local product that cancels the finitely many parameter poles, the kernel in \eqref{eq:resonance expansion}, together with each fixed number of its $s$-derivatives, is bounded on $K\times K$ by
	\[
	 C\bigl(1+|\log u|\bigr)\le C'\bigl(1+|\log r|\bigr).
	\]
	The function on the right is square integrable near the diagonal in two dimensions. The locally pole-cleared kernels therefore depend holomorphically on $s$ in Hilbert--Schmidt norm, and hence in operator norm on $L^2(\mathbb H,d\omega)$. This proves that $T_\chi(s)$ is meromorphic on $\mathbb C$, with possible poles in \eqref{eq:resonance poles}, all of order at most one. Uniqueness follows from the identity theorem for Banach-space-valued holomorphic maps.
	
	It follows from \eqref{eq:resonance expansion} that
	\begin{align*}
		A_{s_0}(z,z')
		&=\operatorname*{Res}_{s=s_0}R_B(s;z,z')\\
		&=-\frac{\Phi_B(z,z')}{4\pi}
		\sum_{k=0}^\infty\frac{u^k}{(k!)^2} \times \operatorname*{Res}_{s=s_0} [(s-|B|)_k(s+|B|)_k \psi(s-|B|+k)\sigma^{-s}]\\
		&\quad -\frac{\Phi_B(z,z')}{4\pi}
		\sum_{k=0}^\infty\frac{u^k}{(k!)^2} \times \operatorname*{Res}_{s=s_0} [(s-|B|)_k(s+|B|)_k \psi(s+|B|+k)\sigma^{-s}].
	\end{align*}
	If $s_0-|B|=-n$ and $s_0+|B|\notin-\mathbb N_0$, only $k\le n$ contribute to the residue, each with coefficient $+1$ from $-\psi(s-|B|+k)$. Hence,
	\begin{align*}
		A_{s_0}(z,z')
		&=-\frac{\Phi_B(z,z')}{4\pi}
		\sum_{k=0}^n \frac{u^k}{(k!)^2} \times \operatorname*{Res}_{s=s_0} [(s-|B|)_k(s+|B|)_k \psi(s-|B|+k)\sigma^{-s}]\\
		&=\frac{\Phi_B(z,z')}{4\pi}
		\sum_{k=0}^n \frac{u^k}{(k!)^2} \times  [(s_0-|B|)_k(s_0+|B|)_k \sigma^{-s_0}]\\
		&=\frac{\Phi_B(z,z')}{4\pi} F(s_0-|B|,s_0+|B|;1;u) \sigma^{-s_0}=\frac{\mu_B(s_0)}{4\pi}\Phi_B(z,z')P^0_{s_0,B}(r).
	\end{align*}
	The other single-family case is similar. If also $s_0+|B|=-m$, both digamma terms contribute. It follows that 
	\begin{align*}
		A_{s_0}(z,z')
		&=\frac{\Phi_B(z,z')}{4\pi} \sigma^{-s_0}[F(s_0-|B|,s_0+|B|;1;u)+F(s_0+|B|,s_0-|B|;1;u)]\\
		&=\frac{\mu_B(s_0)}{4\pi}\Phi_B(z,z')P^0_{s_0,B}(r).
	\end{align*}
	This proves \eqref{eq:resonance residue}. Since $P^0_{s_0,B}(0)=1$ and $\mu_B(s_0)\ge1$, the residue remains nonzero after multiplication by any nonzero $\chi\in C_c^\infty(\mathbb H)$. Thus none of the possible poles cancels, and the poles are precisely those in \eqref{eq:resonance poles}.
\end{proof}

For $0\le n<N$, let $\Pi_n$ denote the orthogonal projection onto $\ker(H_B-E_n(B))$. Since $E_n(B)=E(s_n)$ is an isolated spectral point, the spectral theorem gives, near $s_n$,
\[
 R_B(s)=\frac{\Pi_n}{E_n(B)-E(s)}+G_n(s),
\]
where $G_n$ is holomorphic. Moreover,
\[
 \begin{aligned}
 E_n(B)-E(s)
 &=s_n(1-s_n)-s(1-s)\\
  &=(2s_n-1)(s-s_n)+(s-s_n)^2.
 \end{aligned}
\]
Since $s_n>\frac12$, taking the residue with respect to $s$ yields
\[
 \begin{aligned}
 \operatorname*{Res}_{s=s_n}R_B(s)
 &=\lim_{s\to s_n}(s-s_n)R_B(s)\\
 &=\lim_{s\to s_n}\frac{\Pi_n}{s+s_n-1}
 =\frac{\Pi_n}{2s_n-1}.
 \end{aligned}
\]
Therefore, \eqref{eq:resonance residue} gives
\begin{equation}\label{eq:landau projection}
	A_{s_n}=\frac{\Pi_n}{2s_n-1},\qquad
	\Pi_n(z,z')=\frac{2s_n-1}{4\pi}\Phi_B(z,z')P^0_{s_n,B}(r),
\end{equation}
where the kernel is taken with respect to $d\omega(z')$. As $|B|$ increases through $n+\frac12$, the pole $s=|B|-n$ crosses the threshold and becomes an eigenvalue pole, as illustrated in Figure~\ref{fig:landau-levels}.

\begin{figure}[h]
	\centering
	\begin{tikzpicture}
		\begin{axis}[
			width=10.0cm, height=8.0cm,
			axis lines=middle,
			axis line style={-Stealth, thin},
			clip=false,
			xlabel={$B$}, ylabel={$E$},
			xlabel style={anchor=west, font=\normalsize},
			ylabel style={anchor=south, font=\normalsize},
			xmin=0, xmax=4.0,
			ymin=0, ymax=14,
			xtick={0.5,1.5,2.5,3.5},
			xticklabels={$\frac{1}{2}$,$\frac{3}{2}$,$\frac{5}{2}$,$\frac{7}{2}$},
			ytick=\empty,
			tick label style={font=\normalsize},
			minor tick num=1,
			legend style={at={(0.03,0.97)}, anchor=north west, draw=none, fill=none,
				font=\normalsize, row sep=2.5pt, inner sep=2pt},
			legend cell align=left,
			]
			
			\addplot[blue!70!black, densely dashed, very thick, domain=0:3.71, samples=300] {x^2+0.25};
			\addlegendentry{$E_\ast(B)=B^2+\tfrac14$}
			
			\addplot[red!85!black, thick, domain=0.5:3.8, samples=2] {x};
			\addlegendentry{$E_0(B)=B$}
			\addplot[orange!85!black, thick, domain=1.5:3.8, samples=2] {3*x-2};
			\addlegendentry{$E_1(B)=3B-2$}
			\addplot[green!50!black, thick, domain=2.5:3.8, samples=2] {5*x-6};
			\addlegendentry{$E_2(B)=5B-6$}
			
			\addplot[only marks, mark=*, mark size=1.7pt, black, forget plot]
			coordinates {(1.5,1.5) (2.5,2.5) (2.5,5.5) (3.5,3.5) (3.5,8.5) (3.5,11.5)};
			\addplot[only marks, mark=o, mark size=2.4pt, blue!70!black,
			mark options={fill=white, line width=0.7pt}, forget plot]
			coordinates {(0.5,0.5) (1.5,2.5) (2.5,6.5) (3.5,12.5)};
			
			\node[font=\normalsize, anchor=north east, inner sep=1.5pt]
			at (axis cs:0,0) {$0$};
			
		\end{axis}
	\end{tikzpicture}
	\caption{The hyperbolic Landau levels $E_n(B)$ (solid lines) and the bottom $E_\ast(B)=B^2+\frac14$ of the continuous spectrum (dashed curve).}
	\label{fig:landau-levels}
\end{figure}

Since the pole set is independent of $\chi$, the following definition will make sense.

\begin{definition}[Resolvent resonances]\label{def:resolvent resonances}
	A point $s_0$ with $\operatorname{Re}s_0<\frac12$ is a \emph{resolvent resonance} if it is a pole of the continuation in Proposition~\ref{prop:resonance poles} for one (equivalently, for every) $0\ne\chi\in C_c^\infty(\mathbb H;\mathbb R)$. A pole at $s_0=\frac12$ is a \emph{threshold resolvent resonance}.
\end{definition}

At the threshold, $\frac12\in\mathcal P_B$ exactly when $|B|=m+\frac12$, and then $\mu_B(\frac12)=1$. Hence, the threshold $s=\frac12$ is a pole exactly when $|B|=m+\frac12$, $m\in\mathbb N_0$, in which case
\begin{equation}\label{eq:threshold residue}
	A_{1/2}(z,z')=\frac{\Phi_B(z,z')P^0_{1/2,B}(r)}{4\pi}.
\end{equation}
The threshold resonance at $E_*(B)$ has an effect analogous to the three-dimensional resonance mechanism discussed in \cite[Section~II.A]{Schlag21}. We recall that for $-\Delta+V$ on $\mathbb R^3$, with real $V\in C_c^\infty$, a zero resonance is represented by a nonzero distributional solution $u\notin L^2$ satisfying $(-\Delta+V)u=0$ and $\langle x\rangle^{-\alpha}u\in L^2$ for every $\alpha>\frac12$, where $\langle x\rangle=(1+|x|^2)^{1/2}$; see \cite{JensenKato79} and \cite[Section~II.A]{Schlag21}. At $|B|=m+\frac12$, the function $u_{z'}(z)=A_{1/2}(z,z')$ satisfies $(H_B-E_\ast(B))u_{z'}=0$. By \eqref{eq:hypergeom-vandermonde}, we have
\[
P^0_{1/2,B}(r)
=\sigma^{-1/2}F(-m,m+1;1;1-\sigma^{-1})
\sim2(-1)^m e^{-r/2}\qquad(r\longrightarrow\infty),
\]
which gives $|u_{z'}|^2\sinh r\to2$. Consequently, $u_{z'}\notin L^2(\mathbb H)$ and
\[
 (1+r)^{-\alpha}u_{z'}\in L^2(\mathbb H)
 \quad\Longleftrightarrow\quad\alpha>\tfrac12.
\]
A bound state contributes the nondecaying term $e^{-itE_m(B)}\Pi_mf$; the threshold resonance instead gives a weighted $L^2$ kernel which produces a slowly decaying tail in the absolutely continuous evolution. This will be shown in Theorem~\ref{thm:threshold decay}: at a threshold field, for $f\in C_c^\infty(\mathbb H)$ and $t\ge1$,
\[
\left\|e^{-itH_B}P_{ac}f
-\frac{e^{-itE_\ast(B)}e^{-i\pi/4}}{\sqrt{\pi t}}A_{1/2}f\right\|_{L^\infty}
\le C_Bt^{-3/2}\|f\|_{L^1}.
\]
Thus the stationary profile $A_{1/2}f$ obstructs faster decay when it is excited by the initial data. This contribution survives $P_{ac}$, although $A_{1/2}$ is not an $L^2$ spectral projection.

\section{Heat kernel and Schr\"odinger propagator}
\label{sec:heat-propagator}
In this section, we derive the kernel representations and estimates used in Section~\ref{sec:proof of cor}. We shall frequently use the following elementary estimates, whose proof is omitted.

\begin{lemma}
	\label{lem:hyperbolic-estimates}
	For $\rho>r>0$,
	\begin{equation}
		\label{eq:hyperbolic function1}
		\sinh \rho\;\sim\; \frac{\rho}{1+\rho}\,e^{\rho},
	\end{equation}
	and
	\begin{equation}
		\label{eq:hyperbolic function2}
		\cosh \rho-\cosh r
		=2\sinh\frac{\rho-r}{2}\sinh\frac{\rho+r}{2}
		\;\sim\; \frac{\rho-r}{1+\rho-r}\,\frac{\rho}{1+\rho}\,e^{\rho}.
	\end{equation}
	In particular,
	\begin{equation}
		\label{eq:hyperbolic function3}
		\cosh \rho-\cosh r\;\sim\;
		\begin{cases}
			\displaystyle\frac{\rho^{2}-r^{2}}{1+r}\,e^{r}, & r\leqslant \rho\leqslant r+1,\\[10pt]
			e^{\rho}, & \rho\geqslant r+1 .
		\end{cases}
	\end{equation}
\end{lemma}

\subsection{Heat kernel}
\label{subsec:heat-radial}

We begin with the magnetic heat kernel and the estimates valid for every $B\in\mathbb R$, then treat the Schr\"odinger propagator separately at nonthreshold and threshold fields.

Retain $\Phi_B$, $E_\ast(B)$, $N$ and $s_n$ from Section~\ref{sec:spectral analysis}.
We use the generalized Chebyshev function
\begin{equation}\label{def:T+}
 T_\nu(x)=\cosh(\nu\operatorname{arcosh}x)
 =\tfrac12\sum_\pm(x\pm\sqrt{x^2-1})^\nu,
 \qquad x\ge1,\quad \nu\in\mathbb R.
\end{equation}

\begin{lemma}[{\cite[Theorem~5.1, (5.1) and (5.3)]{AMS01},\cite[Theorem~3.1, (3.1)]{SS20}}]
	\label{lem:heat kernel}
	Let $\tau>0$ and $r=d(z,z')$.  For $\rho>r$, define
	\begin{equation}\label{eq:radial weight}
		W_B(\rho,r)=
		\frac{2T_{2B}\!\left(\cosh(\rho/2)/\cosh(r/2)\right)}
		{\sqrt{2(\cosh\rho-\cosh r)}},
	\end{equation}
	and set $W_B(\rho,r)=0$ for $0\le\rho\le r$. The heat kernel, with respect to $d\omega(z')$, is
	\begin{equation}\label{eq:full heat kernel}
		K^H_\tau(z,z')=
		\frac{\Phi_B(z,z')e^{-\tau E_\ast(B)}}{8\pi^{3/2}\tau^{3/2}}
		\int_r^\infty \rho e^{-\rho^2/(4\tau)}
		W_B(\rho,r)\,d\rho.
	\end{equation}
\end{lemma}

With $W_B$ as in \eqref{eq:radial weight}, define
	\begin{equation}\label{eq:ac radial weight}
		\begin{split}
			S_B(\rho,r)&=2\sum_{n=0}^{N-1}P^0_{s_n,B}(r)\sinh\!\bigl((s_n-\tfrac12)\rho\bigr),\\
			W_B^{ac}(\rho,r)&=W_B(\rho,r)-S_B(\rho,r),\qquad \rho\ge0.
		\end{split}
	\end{equation}
As we shall prove in Theorem \ref{thm:Main}, $W_B^{ac}$ gives the radial part of the kernel of $e^{-itH_B}P_{ac}$. Therefore, it is crucial to estimate $W_B(\rho,r)$ and $S_B(\rho,r)$.

\begin{lemma}
	\label{lem:radial-kernel-bounds}
	Let $B\in\mathbb R$. There exists $C_B>0$ such that, for every $r\ge0$,
	\begin{equation}\label{eq:landau-radial-bound}
	 |P^0_{s_n,B}(r)|\le C_Be^{-s_nr},\qquad 0\le n<N.
	\end{equation}
	For $\rho>r$,
	\begin{align}
	 |W_B(\rho,r)|
	 &\le C_B\frac{e^{|B|(\rho-r)}}{\sqrt{\cosh\rho-\cosh r}}, \label{eq:radial-global-bound}\\
	 |\partial_\rho W_B(\rho,r)|
	 &\le C_B\frac{e^{|B|(\rho-r)}\sinh\rho}{(\cosh\rho-\cosh r)^{3/2}}. \label{eq:radial-global-derivative}
	\end{align}
	If $N\ge1$, then, for $\rho\ge0$,
	\begin{equation}\label{eq:correction-global-bound}
	 |S_B(\rho,r)|+|\partial_\rho S_B(\rho,r)|
	 \le C_Be^{-r/2}
	 \begin{cases}
	 e^{-(s_{N-1}-\frac12)(r-\rho)},&0\le\rho\le r,\\
	 e^{(s_0-\frac12)(\rho-r)},&\rho\ge r.
	 \end{cases}
	\end{equation}
	Moreover, $S_B(0,r)=0$; if $N=0$, then $S_B\equiv0$.
\end{lemma}

\begin{proof}
This lemma is more involved than Lemma \ref{lem:hyperbolic-estimates} but remains elementary.

For $0\le n<N$, we have $s_n=|B|-n>\frac12$. The identities
\[
 1-\tanh^2(r/2)=\cosh^{-2}(r/2),\qquad
 \frac{\tanh^2(r/2)}{\tanh^2(r/2)-1}=-\sinh^2(r/2)
\]
and the Pfaff transformation \eqref{eq:hypergeom-pfaff} give
\begin{align}
	\label{eq:jacobi-hypergeom1}
	P^0_{s_n,B}(r)
	&=\cosh^{-2(|B|-n)}(r/2)
	F\left(-n,2|B|-n;1;\tanh^2(r/2)\right)\notag\\
	&=\cosh^{-2(|B|-n)}(r/2)\cosh^{-2n}(r/2)
	F\left(-n,n+1-2|B|;1;-\sinh^2(r/2)\right)\\
	&=\cosh^{-2|B|}(r/2)P_n^{(0,-2|B|)}(\cosh r),\notag
\end{align}
where the last equality is \eqref{eq:jacobi-hypergeom}. It follows that
\begin{align*}
	|P^0_{s_n,B}(r)|
	&\le \cosh^{-2|B|}(r/2)\cdot  |P_n^{(0,-2|B|)}(\cosh r)|\\
	&\le C_{B,n}\cosh^{-2|B|}(r/2)e^{nr}
	\le C_{B,n}e^{-(|B|-n)r}=C_{B,n}e^{-s_n r},
\end{align*}
since $P_n^{(0,-2|B|)}$ is a polynomial of degree at most $n$. For $N\ge1$, taking the maximum over $0\le n<N$ proves \eqref{eq:landau-radial-bound}.

Next, for $\rho>r$, define $X>0$ by
\[
 \cosh X=\frac{\cosh(\rho/2)}{\cosh(r/2)}
 \le e^{(\rho-r)/2},\qquad
 e^X\le2\cosh X\le2e^{(\rho-r)/2}.
\]
Then $\cosh\rho-\cosh r=2\cosh^2(r/2)\sinh^2X$. Hence, \eqref{def:T+} and \eqref{eq:radial weight} imply 
\[
\begin{gathered}
 W_B(\rho,r)=\frac{\cosh(2|B|X)}{\cosh(r/2)\sinh X},\qquad
 X'(\rho)=\frac{\sinh(\rho/2)}{2\cosh(r/2)\sinh X}.
\end{gathered}
\]
Differentiating gives
\[
\begin{aligned}
 \partial_\rho W_B(\rho,r)
 ={}&\frac{\sinh(\rho/2)}{2\cosh^2(r/2)\sinh^3X}\cdot \left[2|B|\sinh(2|B|X)\sinh X
 -\cosh(2|B|X)\cosh X\right].
\end{aligned}
\]
Since
\[
\begin{aligned}
 &\cosh(2|B|X)\le e^{2|B|X}\le2^{2|B|}e^{|B|(\rho-r)},\\
 &\left\{2|B|\sinh(2|B|X)\sinh X
 -\cosh(2|B|X)\cosh X\right\}\\
 &\qquad\le(2|B|+1)e^{2|B|X}\cosh X
 \le C_Be^{|B|(\rho-r)}\cosh X,
\end{aligned}
\]
we obtain
\[
\begin{aligned}
 |W_B(\rho,r)|
 &\le\frac{C_Be^{|B|(\rho-r)}}{\cosh(r/2)\sinh X}
 \le\frac{C_Be^{|B|(\rho-r)}}{\sqrt{\cosh\rho-\cosh r}},\\
 |\partial_\rho W_B(\rho,r)|
 &\le\frac{C_Be^{|B|(\rho-r)}\sinh(\rho/2)\cosh X}
 {\cosh^2(r/2)\sinh^3X}\\
 &\le\frac{C_Be^{|B|(\rho-r)}\sinh(\rho/2)\cosh(\rho/2)}
 {(\cosh\rho-\cosh r)^{3/2}}
 \le\frac{C_Be^{|B|(\rho-r)}\sinh\rho}{(\cosh\rho-\cosh r)^{3/2}}.
\end{aligned}
\]
This proves \eqref{eq:radial-global-bound} and \eqref{eq:radial-global-derivative}.

\smallskip

It remains to estimate $S_B$. The definition gives $S_B(0,r)=0$ and
\[
 \partial_\rho S_B(\rho,r)
 =2\sum_{n=0}^{N-1}(s_n-\tfrac12)P^0_{s_n,B}(r)
 \cosh\!\bigl((s_n-\tfrac12)\rho\bigr).
\]
For $N\ge1$ and $\rho\ge0$, \eqref{eq:landau-radial-bound} gives
\[
\begin{aligned}
 |S_B(\rho,r)|+|\partial_\rho S_B(\rho,r)|
 &\le C_B\sum_{n=0}^{N-1}e^{-s_nr+(s_n-\frac12)\rho}\\
 &=C_Be^{-r/2}\sum_{n=0}^{N-1}e^{-(s_n-\frac12)(r-\rho)}.
\end{aligned}
\]
Since $s_{N-1}\le s_n\le s_0$,
\[
 \sum_{n=0}^{N-1}e^{-(s_n-\frac12)(r-\rho)}
 \le N\begin{cases}
 e^{-(s_{N-1}-\frac12)(r-\rho)},&0\le\rho\le r,\\
 e^{(s_0-\frac12)(\rho-r)},&\rho\ge r.
 \end{cases}
\]
This proves \eqref{eq:correction-global-bound}.
\end{proof}

\begin{remark}
	Note \eqref{eq:radial-global-bound} and \eqref{eq:radial-global-derivative} quantify the singular behavior of  $W_B(\rho,r)$ as $\rho\to r^+$. Indeed, on $r<\rho\le r+1$, we have $e^{|B|(\rho-r)}\le e^{|B|}$; hence \eqref{eq:hyperbolic function3}, \eqref{eq:radial-global-bound} and \eqref{eq:radial-global-derivative} give 
	\begin{align}
		|W_B(\rho,r)|&\le C_B\frac{(1+r)^{1/2}e^{-r/2}}{\sqrt{\rho^2-r^2}}, \label{eq:ac-near-bound}\\
		|\partial_\rho W_B(\rho,r)|
		&\le C_B\frac{\sinh\rho}{(\cosh\rho-\cosh r)^{3/2}}. \label{eq:ac-near-derivative}
	\end{align}
\end{remark}

\begin{remark}
	We mention that \eqref{eq:ac radial weight} and \eqref{eq:jacobi-hypergeom1}  give the following useful formula:
	\begin{equation}\label{eq:ac-jacobi-correction}
		\begin{aligned}
			S_B(\rho,r)
			={}&\frac{2}{\cosh^{2|B|}(r/2)}
			\sum_{n=0}^{N-1}P_n^{(0,-2|B|)}(\cosh r)
			\sinh\!\bigl((|B|-n-\tfrac12)\rho\bigr)\\
			={}&\frac{e^{(|B|-\frac12)\rho}}{\cosh^{2|B|}(r/2)}
			\sum_{n=0}^{N-1}P_n^{(0,-2|B|)}(\cosh r)q^n
			-\sum_{n=0}^{N-1}P^0_{s_n,B}(r)e^{-(|B|-n-\frac12)\rho},
		\end{aligned}
	\end{equation}
	where $q=e^{-\rho}$ and we use $2\sinh v=e^v-e^{-v}$.
\end{remark}

\subsection{Nonthreshold fields}
\label{subsec:nonthreshold-kernel}
Throughout this subsection, $B\notin\frac12+\mathbb Z$. We first estimate the continuous radial weight, then obtain the propagator kernel by a boundary value in the time parameter.

\begin{lemma}[Nonthreshold estimates]\label{lem:ac-kernel-bounds}
	Assume $B\notin\frac12+\mathbb Z$. There exist $C_B>0$ and $\varepsilon_B>0$ such that, for every $r\ge0$,
	\begin{equation}\label{eq:ac-tail-bound}
	 |W_B^{ac}(\rho,r)|+|\partial_\rho W_B^{ac}(\rho,r)|
	 \le C_Be^{-r/2}\frac{e^{-\varepsilon_B(\rho-r)}}{(1-e^{-(\rho-r)})^2},\qquad \rho>r.
	\end{equation}
\end{lemma}

\begin{proof}
	For $\rho>r>0$, put $h=e^{-r}\in(0,1)$, $u=e^{r-\rho}$ and $q=hu$, so that $0<u<1$. Write
	\[
	\begin{aligned}
		D_h(u)&=(1-u)(1-h^2u),&
		A_h(u)&=\frac{1+hu+\sqrt{D_h(u)}}2.
	\end{aligned}
	\]
	All square roots below are positive. By \eqref{eq:jacobi-generating},
	\begin{equation}\label{eq:ac-jacobi-generating}
	 \frac{A_h(u)^{2|B|}}{\sqrt{D_h(u)}}
	 =\sum_{n=0}^{\infty}P_n^{(0,-2|B|)}(\cosh r)q^n,
	 \qquad 0<q<e^{-r}.
	\end{equation}
	With $x=\cosh(\rho/2)/\cosh(r/2)$, we have $
	\sqrt{2(\cosh\rho-\cosh r)}=q^{-1/2}\sqrt{D_h(u)}$ and
	\[
	\begin{aligned}
	 x+\sqrt{x^2-1}&=\frac{A_h(u)}{\cosh(r/2)\sqrt q},\qquad  x-\sqrt{x^2-1}&=\frac{\cosh(r/2)\sqrt q}{A_h(u)}.
	\end{aligned}
	\]
	Substitution into \eqref{eq:radial weight} yields
	\begin{equation}\label{eq:ac-two-branches}
	\begin{aligned}
	 W_B(\rho,r)
	 ={}&\cosh^{-2|B|}(r/2)e^{(|B|-\frac12)\rho}
	 F_{+,h}(u)+\cosh^{2|B|}(r/2)e^{-(|B|+\frac12)\rho}
	 F_{-,h}(u),
	\end{aligned}
	\end{equation}
	where
	\[ F_{\pm,h}(u)=\frac{A_h(u)^{\pm2|B|}}{\sqrt{D_h(u)}}.\]
By \eqref{eq:ac-jacobi-generating}, for $h=e^{-r}>0$,
\[
 F_{+,h}(u)=\sum_{n=0}^{\infty}d_n(h)u^n,
 \qquad d_n(h)=h^nP_n^{(0,-2|B|)}(\cosh r).
\]
Since
\[
 h\frac{\cosh r-1}{2}=\left(\frac{1-h}{2}\right)^2,
 \qquad
 h\frac{\cosh r+1}{2}=\left(\frac{1+h}{2}\right)^2,
\]
the binomial formula \eqref{eq:jacobi-binomial} gives
\[
 d_n(h)=\sum_{j=0}^{n}\binom nj\binom{n-2|B|}{j}
 \left(\frac{1-h}{2}\right)^{2j}
 \left(\frac{1+h}{2}\right)^{2(n-j)}.
\]
For $0\le j\le n$,
\[
 \frac{\left|\binom{n-2|B|}{j}\right|}{\binom nj}
 =\prod_{k=n-j+1}^{n}\left|1-\frac{2|B|}{k}\right|
 \le\prod_{1\le k<|B|}\left(\frac{2|B|}{k}-1\right)
 \le C_B,
\]
because each factor with $k\ge|B|$ is at most $1$; empty products are understood as $1$. The binomial theorem therefore yields
\begin{equation}\label{eq:ac-coefficient-bound}
\begin{aligned}
 |d_n(h)|
 &\le C_B\sum_{j=0}^{n}
 \left[\binom nj\left(\frac{1-h}{2}\right)^j
 \left(\frac{1+h}{2}\right)^{n-j}\right]^2\\
 &\le C_B\left[\sum_{j=0}^{n}\binom nj
 \left(\frac{1-h}{2}\right)^j
 \left(\frac{1+h}{2}\right)^{n-j}\right]^2
 =C_B.
\end{aligned}
\end{equation}
This bound is uniform in $n\in\mathbb N_0$ and $0<h<1$.

Define the Taylor remainder
\[
 R_{N,h}(u):=F_{+,h}(u)-\sum_{n=0}^{N-1}d_n(h)u^n
 =\sum_{n=N}^{\infty}d_n(h)u^n.
\]
Subtracting \eqref{eq:ac-jacobi-correction} from \eqref{eq:ac-two-branches} now gives the exact decomposition
\begin{equation}\label{eq:ac-remainder-decomposition}
\begin{aligned}
 W_B^{ac}(\rho,r)
 ={}&\cosh^{-2|B|}(r/2)e^{(|B|-\frac12)\rho}R_{N,h}(u)\\
 &+\cosh^{2|B|}(r/2)e^{-(|B|+\frac12)\rho}F_{-,h}(u)+\sum_{n=0}^{N-1}P^0_{s_n,B}(r)e^{-(s_n-\frac12)\rho}.
\end{aligned}
\end{equation}

By \eqref{eq:ac-coefficient-bound}, the series and its differentiated series converge uniformly on every compact subinterval of $[0,1)$. Summing the geometric series gives
\[
\begin{aligned}
 |R_{N,h}(u)|
 &\le C_B\sum_{n=N}^{\infty}u^n
 =C_B\frac{u^N}{1-u},\\
 |uR_{N,h}'(u)|
 &\le C_B\sum_{n=N}^{\infty}nu^n=C_Bu^N
 \left(\frac{N}{1-u}+\frac{u}{(1-u)^2}\right).
\end{aligned}
\]
Since $0<u<1$ and $N$ depends only on $B$,
\[
 |R_{N,h}(u)|+|uR_{N,h}'(u)|\le C_B\frac{u^N}{(1-u)^2}.
\]
To estimate $F_{-,h}$ directly, note that, for $0<u<1$,
\[
 (1-u)^2\le D_h(u)\le1,\qquad
 \tfrac12\le A_h(u)\le2,
\]
and
\[
\begin{aligned}
 D_h'(u)&=-(1+h^2)+2h^2u,& |D_h'(u)|&\le2,\\
 A_h'(u)&=\frac h2+\frac{D_h'(u)}{4\sqrt{D_h(u)}},&
 |A_h'(u)|&\le\frac12+\frac1{2(1-u)}\le\frac1{1-u}.
\end{aligned}
\]
Moreover,
\[
 \frac{D_h'(u)}{D_h(u)}=-\frac1{1-u}-\frac{h^2}{1-h^2u},
 \qquad\left|\frac{D_h'(u)}{D_h(u)}\right|\le\frac2{1-u}.
\]
Consequently, $|F_{-,h}(u)|\le 2^{2|B|}/(1-u)$ and
\[
 F_{-,h}'(u)=F_{-,h}(u)
 \left(-2|B|\frac{A_h'(u)}{A_h(u)}-\frac{D_h'(u)}{2D_h(u)}\right),
\]
and these bounds imply
\begin{equation}\label{eq:ac-negative-branch-bound}
 |F_{-,h}(u)|+|uF_{-,h}'(u)|\le\frac{C_B}{(1-u)^2},
 \qquad0<u<1,\quad0<h<1.
\end{equation}

Here $h$ is fixed when differentiating in $\rho$, and $\partial_\rho u=-u$. In particular,
\[
\begin{aligned}
 \partial_\rho\!\left[e^{(|B|-\frac12)\rho}R_{N,h}(u)\right]
 &=e^{(|B|-\frac12)\rho}
 \left[(|B|-\tfrac12)R_{N,h}(u)-uR_{N,h}'(u)\right],\\
 \partial_\rho\!\left[e^{-(|B|+\frac12)\rho}F_{-,h}(u)\right]
 &=e^{-(|B|+\frac12)\rho}
 \left[-(|B|+\tfrac12)F_{-,h}(u)-uF_{-,h}'(u)\right].
\end{aligned}
\]
Using $e^{r/2}/2\le\cosh(r/2)\le e^{r/2}$, we obtain
\[
\begin{aligned}
 \sum_{j=0}^{1}\left|\partial_\rho^j
 \left[\cosh^{-2|B|}(r/2)e^{(|B|-\frac12)\rho}R_{N,h}(u)\right]\right|
 &\le\frac{C_B}{(1-u)^2}\cosh^{-2|B|}(r/2)e^{(|B|-\frac12)\rho}u^N\\
 &\le\frac{C_B}{(1-u)^2}e^{-|B|r+(|B|-\frac12)\rho+N(r-\rho)}\\
 &=\frac{C_B}{(1-u)^2}e^{-r/2}e^{-(N+\frac12-|B|)(\rho-r)},\\[3pt]
 \sum_{j=0}^{1}\left|\partial_\rho^j
 \left[\cosh^{2|B|}(r/2)e^{-(|B|+\frac12)\rho}F_{-,h}(u)\right]\right|
 &\le\frac{C_B}{(1-u)^2}e^{|B|r-(|B|+\frac12)\rho}\\
 &=\frac{C_B}{(1-u)^2}e^{-r/2}e^{-(|B|+\frac12)(\rho-r)}.
\end{aligned}
\]
For $N\ge1$, the last term of \eqref{eq:ac-remainder-decomposition} satisfies
\[
\begin{aligned}
 \sum_{j=0}^{1}\left|\partial_\rho^j
 \sum_{n=0}^{N-1}P^0_{s_n,B}(r)e^{-(s_n-\frac12)\rho}\right|
 &\le\sum_{n=0}^{N-1}(1+s_n-\tfrac12)
 |P^0_{s_n,B}(r)|e^{-(s_n-\frac12)\rho}\\
 &\le C_B\sum_{n=0}^{N-1}e^{-s_nr-(s_n-\frac12)\rho}
 =C_Be^{-r/2}\sum_{n=0}^{N-1}e^{-(s_n-\frac12)(\rho+r)}\\
 &\le C_Be^{-r/2}e^{-(|B|-N+\frac12)(\rho-r)}.
\end{aligned}
\]
Under the nonthreshold assumption, $N-1<|B|-\frac12<N$ when $N\ge1$, so the rates $N+\frac12-|B|$ and $|B|-N+\frac12$ are positive. When $N=0$, we have $|B|<\frac12$ and the finite sum is absent. Since $(1-u)^{-2}\ge1$ and $1-u=1-e^{-(\rho-r)}$, combining these estimates proves \eqref{eq:ac-tail-bound} with
\[
 \varepsilon_B=
 \begin{cases}
 \frac12-|B|,&N=0,\\
 \min\{N+\frac12-|B|,\ |B|-N+\frac12\},&N\ge1.
 \end{cases}
\]
The estimate at $r=0$ follows by continuity.
\end{proof}

\begin{theorem}
	\label{thm:Main}
	Assume $B\notin\frac12+\mathbb Z$ and let $W_B^{ac}$ be defined by \eqref{eq:ac radial weight}.
	For $u_0\in C_c^\infty(\mathbb H)$, the solution $u(t)=e^{-itH_B}P_{ac}u_0$ of \eqref{eq:prob}, with initial datum $u(0)=P_{ac}u_0$, is given, for $t\ne0$, by
	\begin{equation}
		\label{eq:magnetic solution}
		u(t,z)=\int_{\mathbb{H}}K_B^{ac}(t;z,z')\,u_0(z')
		\,d\omega(z'),
	\end{equation}
	where
	\[
	K_B^{ac}(t;z,z')=c_{\operatorname{sgn}(t)}|t|^{-\frac{3}{2}}e^{-itE_\ast(B)}\cdot \Phi_B(z,z')\int_{0}^{\infty}e^{i\frac{\rho^2}{4t}}W_B^{ac}(\rho,r)\rho\,d\rho.
	\]
	Here $r=d(z,z')$, \(c_\pm=(8\pi^{3/2})^{-1}e^{\mp3\pi i/4}\), and $\Phi_B$ is the phase in \eqref{eq:magnetic phase}.
\end{theorem}

\begin{remark}
	If \(B=0\), we have
	\[
	W_0(\rho,r)=\frac{2}{\sqrt{2(\cosh\rho-\cosh r)}},\qquad \rho>r,
	\]
	and \eqref{eq:magnetic solution} reduces to the propagator formula for the free hyperbolic Laplacian \(-\Delta_{\mathbb{H}}=-y^2(\partial_x^2+\partial_y^2)\), previously derived by hyperbolic Fourier analysis in~\cite{Banica2007}.
\end{remark}

\begin{proof}
	Let $\tau>0$. Differentiating the elementary Gaussian integral gives
	\[
	\int_0^\infty\rho e^{-\rho^2/(4\tau)}\sinh\!\bigl((s_n-\tfrac12)\rho\bigr)\,d\rho
	=2\sqrt\pi\,\tau^{3/2}(s_n-\tfrac12)
	e^{\tau(s_n-\tfrac12)^2}.
	\]
	It follows from this formula and \eqref{eq:landau projection} that
	\begin{align*}
		\sum_{n=0}^{N-1} e^{-\tau E_n(B)} \Pi_n(z,z')
		&=\sum_{n=0}^{N-1} e^{-\tau E_n(B)}\cdot  \frac{2s_n-1}{4\pi}\Phi_B(z,z')P^0_{s_n,B}(r)\\
		&=\sum_{n=0}^{N-1} e^{-\tau[E_\ast(B)-(s_n-\tfrac12)^2]}
		\frac{s_n-\frac{1}{2}}{2\pi}\Phi_B(z,z')P^0_{s_n,B}(r)\\
		&=\frac{e^{-\tau E_\ast(B)}
			\Phi_B(z,z')}{4(\pi\tau)^{3/2}} \int_0^\infty\rho e^{-\rho^2/(4\tau)} \sum_{n=0}^{N-1} \sinh\!\bigl((s_n-\tfrac12)\rho\bigr)  P^0_{s_n,B}(r)\,d\rho
	\end{align*}
	Subtracting these integrals over $[0,\infty)$ from the heat kernel in Lemma~\ref{lem:heat kernel} gives the kernel of $e^{-\tau H_B}P_{ac}$:
	\begin{equation}\label{eq:ac heat kernel}
	 K^{H,ac}_\tau(z,z')=
	 \frac{\Phi_B(z,z')e^{-\tau E_\ast(B)}}{8\pi^{3/2}\tau^{3/2}}
	 \int_0^\infty \rho e^{-\rho^2/(4\tau)}W_B^{ac}(\rho,r)\,d\rho.
	\end{equation}
	This identity holds for every $B\in\mathbb R$, including threshold fields.
	
	\bigskip
	
	To pass to the boundary in time, we use the spectral theorem and dominated convergence. Thus, we need to control the integral $\int_0^{\infty} \rho |W_B^{ac}(\rho,r)|\,d\rho$. Using Lemmas~\ref{lem:radial-kernel-bounds} and~\ref{lem:ac-kernel-bounds}, we split the integral as follows:
	\[\int_0^{\infty} \rho |W_B^{ac}(\rho,r)|\,d\rho=\int_0^r+\int_r^{r+1}+\int_{r+1}^{\infty}.\]
	
	On $0\le\rho<r$, we have $W_B^{ac}=-S_B$. If $N\ge1$, \eqref{eq:correction-global-bound} gives
	\[
	\begin{aligned}
	 \int_0^r\rho|W_B^{ac}(\rho,r)|\,d\rho
	 &\le C_Br e^{-r/2}\int_0^r
	 e^{-(s_{N-1}-\frac12)(r-\rho)}\,d\rho\\
	 &=C_Br e^{-r/2}
	 \frac{1-e^{-(s_{N-1}-\frac12)r}}{s_{N-1}-\frac12}
	 \le C_Br e^{-r/2}.
	\end{aligned}
	\]
	On $[r,r+1]$, \eqref{eq:ac-near-bound} and \eqref{eq:correction-global-bound} give
	\[
	\begin{aligned}
	 \int_r^{r+1}\rho|W_B^{ac}(\rho,r)|\,d\rho
	 &\le\int_r^{r+1}\rho\bigl(|W_B(\rho,r)|+|S_B(\rho,r)|\bigr)\,d\rho\\
	 &\le C_Be^{-r/2}\left[
	 \sqrt{1+r}\int_r^{r+1}\frac{\rho\,d\rho}{\sqrt{\rho^2-r^2}}
	 +\int_r^{r+1}\rho\,d\rho\right]\\
	 &=C_Be^{-r/2}\left[\sqrt{(1+r)(2r+1)}+r+\tfrac12\right]\\
	 &\le C_B(1+r)e^{-r/2}.
	\end{aligned}
	\]
	Finally, since $(1-e^{-(\rho-r)})^{-2}\le(1-e^{-1})^{-2}$ for $\rho\ge r+1$, \eqref{eq:ac-tail-bound} implies
	\[
	\begin{aligned}
	 \int_{r+1}^{\infty}\rho|W_B^{ac}(\rho,r)|\,d\rho
	 &\le C_Be^{-r/2}\int_{r+1}^{\infty}
	 \rho e^{-\varepsilon_B(\rho-r)}\,d\rho\\
	 &=C_Be^{-r/2}e^{-\varepsilon_B}
	 \left(\frac{r+1}{\varepsilon_B}+\frac1{\varepsilon_B^2}\right)\\
	 &\le C_B(1+r)e^{-r/2}.
	\end{aligned}
	\]
	Adding the three integrals yields
	\begin{equation}\label{eq:ac-moment-bound}
	 \int_0^\infty\rho|W_B^{ac}(\rho,r)|\,d\rho
	 \le C_B(1+r)e^{-r/2}.
	\end{equation}
	
	Now fix $t\ne0$ and set $\tau=\varepsilon+it$ with $\varepsilon>0$. Since
	\[
	 |e^{-\rho^2/(4\tau)}|=\exp\left(-\frac{\varepsilon\rho^2}{4(\varepsilon^2+t^2)}\right)\le1,
	 \qquad
	 \lim_{\varepsilon\downarrow0}\tau^{-3/2}
	 =|t|^{-3/2}e^{-3\pi i\operatorname{sgn}(t)/4},
	\]
	dominated convergence gives the displayed formula for $K_B^{ac}$.
\end{proof}

We now estimate the radial oscillatory integral
\begin{equation}
		\label{eq:L_t}
		L_t(r):=\int_{0}^{\infty}\rho\,e^{i\frac{\rho^{2}}{4t}}W_B^{ac}(\rho,r)\,d\rho .
	\end{equation}
Theorem~\ref{thm:Main} gives
\[
 K_B^{ac}(t;z,z')
 =c_{\operatorname{sgn}(t)}|t|^{-3/2}e^{-itE_\ast(B)}
 \Phi_B(z,z')L_t(r).
\]
The following proposition provides the pointwise and Lorentz bounds used in Section~\ref{sec:proof of cor}.

\begin{proposition}
		\label{pro:upper bound}
		Assume $B\notin\frac12+\mathbb Z$. There exists a constant $C_B>0$ such that, for every $t\in\mathbb{R}^*$ and $r\geqslant 0$,
		\begin{equation}
			\label{eq:pointwise estimate}
			|L_t(r)|\leqslant C_B \min\!\Bigl\{1+r,\;\sqrt{|t|}\,(1+r)^{1/2}\Bigr\}\, e^{-r/2}.
		\end{equation}
	Moreover, for $2<q<\infty$ and $1\le\alpha\le\infty$,
\begin{equation}
			\label{eq:Lorentz norm estimate}
			\| L_t \|_{L^{q,\alpha}} \leqslant C_{B,q,\alpha} \begin{cases} 
				|t|^{1/2}, & \text{if } 0 < |t| \leqslant 1, \\
				1, & \text{if } |t| \geqslant 1.
			\end{cases}
		\end{equation}
\end{proposition}
	
\begin{proof}
	Since $L_t$ is even in $B$ and $L_{-t}=\overline{L_t}$, it is sufficient to consider $B\geqslant0$, $B\notin\frac12+\mathbb N_0$, and $t>0$. First, \eqref{eq:ac-moment-bound} gives
\[
|L_t(r)|\leqslant C_B(1+r)e^{-r/2}
\]
for every $t>0$. In particular, this proves the required bound when $t\geqslant1+r$.

For $0<t\leqslant1+r$, put $A=\sqrt{r^2+t}\leqslant r+1$ and decompose
\begin{align*}
L_t(r)
={}&\int_r^A\rho e^{i\rho^2/(4t)}W_B(\rho,r)\,d\rho
+\int_A^{r+1}\rho e^{i\rho^2/(4t)}W_B(\rho,r)\,d\rho\\
&-\int_0^{r+1}\rho e^{i\rho^2/(4t)}S_B(\rho,r)\,d\rho
+\int_{r+1}^\infty\rho e^{i\rho^2/(4t)}W_B^{ac}(\rho,r)\,d\rho\\
={}&I_0+I_1-J+T.
\end{align*}
By \eqref{eq:ac-near-bound},
\[
|I_0|\leqslant C_B\sqrt{1+r}\,e^{-r/2}
\int_r^A\frac{\rho}{\sqrt{\rho^2-r^2}}\,d\rho
=C_B\sqrt{t}\sqrt{1+r}\,e^{-r/2}.
\]
For the remaining integrals, we use
\[
\rho e^{i\rho^2/(4t)}=-2it\,\frac{d}{d\rho}e^{i\rho^2/(4t)}.
\]
First, \eqref{eq:ac-near-derivative} and \eqref{eq:hyperbolic function3} give
\[
\begin{aligned}
 \int_A^{r+1}|\partial_\rho W_B(\rho,r)|\,d\rho
 &\le C_B\int_A^{r+1}
 \frac{\sinh\rho}{(\cosh\rho-\cosh r)^{3/2}}\,d\rho\\
 &=2C_B\left[\frac1{\sqrt{\cosh A-\cosh r}}
 -\frac1{\sqrt{\cosh(r+1)-\cosh r}}\right]\\
 &\le C_B\frac{(1+r)^{1/2}e^{-r/2}}{\sqrt{A^2-r^2}}.
\end{aligned}
\]
Together with \eqref{eq:ac-near-bound}, integration by parts yields
\begin{align*}
|I_1|
&\leqslant2t\left(|W_B(A,r)|+|W_B(r+1,r)|
+\int_A^{r+1}|\partial_\rho W_B(\rho,r)|\,d\rho\right)\\
&\leqslant C_Bt\,\frac{\sqrt{1+r}\,e^{-r/2}}{\sqrt{A^2-r^2}}
=C_B\sqrt{t}\sqrt{1+r}\,e^{-r/2}.
\end{align*}
If $N\ge1$, then $s_{N-1}-\frac12>0$, and \eqref{eq:correction-global-bound} gives
\[
\begin{aligned}
	 \int_0^{r+1}|\partial_\rho S_B(\rho,r)|\,d\rho
 &\le C_Be^{-r/2}\left[\int_0^{r}
 e^{-(s_{N-1}-\frac12)(r-\rho)}\,d\rho+\int_r^{r+1}
	 e^{(s_0-\frac12)(\rho-r)}\,d\rho\right]\\
 &=C_Be^{-r/2}
\left[ \frac{1-e^{-(s_{N-1}-\frac12)r}}{s_{N-1}-\frac12}+\frac{e^{s_0-\frac12}-1}{s_0-\frac{1}{2}}\right]
 \le C_Be^{-r/2}.
\end{aligned}
\]
Also $|S_B(r+1,r)|\le C_Be^{-r/2}$ and $S_B(0,r)=0$, so
\[
|J|\leqslant2t\left(|S_B(r+1,r)|
+\int_0^{r+1}|\partial_\rho S_B(\rho,r)|\,d\rho\right)
\leqslant C_Bt e^{-r/2}.
\]
For $N=0$, $J=0$. On $\rho\ge r+1$, the factor $(1-e^{-(\rho-r)})^{-2}$ in \eqref{eq:ac-tail-bound} is bounded. The resulting exponential bound makes the boundary term at infinity vanish and gives
\begin{align*}
|T|
&\leqslant2t\left(|W_B^{ac}(r+1,r)|
+\int_{r+1}^\infty|\partial_\rho W_B^{ac}(\rho,r)|\,d\rho\right)\\
&\leqslant C_Bt e^{-r/2}.
\end{align*}
Since $t\leqslant\sqrt{t}\sqrt{1+r}$ in the present range, these four estimates imply
\[
|L_t(r)|\leqslant C_B\sqrt{t}\sqrt{1+r}\,e^{-r/2}.
\]
Combining the two time ranges proves \eqref{eq:pointwise estimate}.

To estimate the Lorentz norms, recall that if $f$ is a positive radial decreasing function on $\mathbb{H}$, then its decreasing rearrangement is $f^* = f \circ V^{-1}$, where
		\[
		V(r) = C \int_0^r \sinh s\,ds \sim \begin{cases} 
			r^2 & \text{as } r \to 0, \\
			e^{r} & \text{as } r \to +\infty
		\end{cases}
		\]
		is the volume of a ball of radius $r > 0$ in $\mathbb{H}$. Hence, we have, for $1 \leqslant \alpha < \infty$, 
		\begin{align*}
			\| f \|_{L^{q,\alpha}}
			&= \Bigl[ \int_0^{+\infty} \bigl( V(r)^{1/q} f(r) \bigr)^\alpha \frac{V'(r)}{V(r)}\,dr \Bigr]^{1/\alpha}\\
			&\asymp \Bigl[ \int_0^1 f(r)^\alpha r^{\frac{2\alpha}{q} - 1}\,dr \Bigr]^{1/\alpha}
			+ \Bigl[ \int_1^{+\infty} f(r)^\alpha e^{\frac{\alpha}{q} r}\,dr \Bigr]^{1/\alpha},
		\end{align*}
		and for $\alpha = \infty$,
		\[
		\| f \|_{L^{q,\infty}} = \sup_{r > 0} V(r)^{1/q} f(r)
		\sim \sup_{0 < r < 1} r^{\frac{2}{q}} f(r) + \sup_{r \geqslant 1} e^{\frac{1}{q} r} f(r).
		\]
The pointwise bound \eqref{eq:pointwise estimate} implies
\[
 |L_t(r)|\le C_B
 \begin{cases}
 \sqrt{|t|}\,(1+r)^{1/2}e^{-r/2},&0<|t|\le1,\\
 (2+r)e^{-r/2},&|t|\ge1.
 \end{cases}
\]
Both spatial profiles are decreasing on $[0,\infty)$. Since $|L_t|$ is bounded by the corresponding radial decreasing majorant, applying the preceding formulas to the majorants proves \eqref{eq:Lorentz norm estimate}.
\end{proof}

\subsection{Threshold fields}
\label{subsec:threshold decay}

At threshold fields $|B|=m+\frac12$, $W_B^{ac}(\rho,r)$ tends to $P^0_{1/2,B}(r)$ as $\rho\to\infty$, rather than decaying to zero. This profile corresponds to the threshold residue $A_{1/2}$ in \eqref{eq:threshold residue}; it is not an $L^2$ eigenprojection and is therefore not removed by $P_{ac}$. In the heat-kernel representation, its contribution is computed exactly by
\[
 \int_0^\infty\rho e^{-\rho^2/(4\tau)}\,d\rho=2\tau,
 \qquad \operatorname{Re}\tau>0.
\]
Consequently, the prefactor $\tau^{-3/2}$ produces a $\tau^{-1/2}$ term, whose boundary value gives the slower $|t|^{-1/2}$ decay. After subtracting this profile, the remaining radial weight has an exponentially decaying tail and can be treated by the same moment and oscillatory-integral arguments as in Section~\ref{subsec:nonthreshold-kernel}. We make this separation explicit below.

\begin{lemma}[Threshold estimates]\label{lem:threshold-kernel-bounds}
	Let $|B|=m+\frac12$, where $m\in\mathbb N_0$. For $\rho,r\ge0$, define
	\begin{equation}\label{eq:threshold radial remainder}
	V_B(\rho,r):=W_B^{ac}(\rho,r)-P^0_{1/2,B}(r).
	\end{equation}
	The threshold profile is
	\begin{equation}\label{eq:threshold radial profile}
	\begin{aligned}
	P^0_{1/2,B}(r)
	&=\frac{1}{\cosh(r/2)}
	F\left(-m,m+1;1;\tanh^2(r/2)\right)\\
	&=\frac{P_m(1-2\tanh^2(r/2))}{\cosh(r/2)},
	\end{aligned}
	\end{equation}
	where $P_m$ is the Legendre polynomial of degree $m$ defined in \eqref{eq:legendre-hypergeom-bound}. Moreover,
	\begin{equation}\label{eq:threshold profile bound}
	|P^0_{1/2,B}(r)|\le2e^{-r/2},
	\qquad P^0_{1/2,B}(0)=1.
	\end{equation}
	There exists $C_B>0$ such that, for every $r\ge0$,
	\begin{equation}\label{eq:threshold radial tail}
	|V_B(\rho,r)|+|\partial_\rho V_B(\rho,r)|
	\le C_Be^{-r/2}e^{-(\rho-r)},\qquad \rho\ge r+1.
	\end{equation}
\end{lemma}

\begin{proof}
	At the present field, $N=m,s_n-\tfrac12=m-n\ge1\quad(0\le n<m)$. Thus \eqref{eq:P,Q,snB} and \eqref{eq:legendre-hypergeom-bound} give \eqref{eq:threshold radial profile}. Since
\[
 -1\le1-2\tanh^2(r/2)\le1,\qquad
 |P_m(x)|\le1\quad(-1\le x\le1),\qquad P_m(1)=1,
\]
we obtain \eqref{eq:threshold profile bound}. The Pfaff transformation also gives
\[
\begin{aligned}
 P^0_{1/2,B}(r)
 &=\cosh^{-1}(r/2)F\left(-m,m+1;1;\tanh^2(r/2)\right)\\
 &=\cosh^{-(2m+1)}(r/2)F\left(-m,-m;1;-\sinh^2(r/2)\right)\\
 &=\cosh^{-(2m+1)}(r/2)P_m^{(0,-2m-1)}(\cosh r).
\end{aligned}
\]
This is the coefficient with index $n=m$, immediately after the $m$ Landau terms.

\smallskip
Use $h=e^{-r}$, $u=e^{r-\rho}$ and $F_{\pm,h}$ from the proof of Lemma~\ref{lem:ac-kernel-bounds}. In particular,
\[
 F_{+,h}(u)=\sum_{n=0}^{\infty}d_n(h)u^n,
 \qquad d_n(h)=h^nP_n^{(0,-2m-1)}(\cosh r).
\]
Define the remainder after the threshold term by
\[
 R_{m+1,h}(u):=F_{+,h}(u)-\sum_{n=0}^{m}d_n(h)u^n
 =\sum_{n=m+1}^{\infty}d_n(h)u^n.
\]
Since $R_{m,h}(u)=d_m(h)u^m+R_{m+1,h}(u)$ and $e^{m\rho}h^mu^m=1$,
\[
\begin{aligned}
 &\cosh^{-(2m+1)}(r/2)e^{m\rho}d_m(h)u^m\\
 &\qquad=\cosh^{-(2m+1)}(r/2)P_m^{(0,-2m-1)}(\cosh r)
 =P^0_{1/2,B}(r).
\end{aligned}
\]
Subtracting this term from \eqref{eq:ac-remainder-decomposition} yields
\begin{equation}\label{eq:threshold remainder decomposition}
\begin{aligned}
 V_B(\rho,r)
 ={}&\cosh^{-(2m+1)}(r/2)e^{m\rho}R_{m+1,h}(u)\\
 &+\cosh^{2m+1}(r/2)e^{-(m+1)\rho}F_{-,h}(u)\\
 &+\sum_{n=0}^{m-1}P^0_{s_n,B}(r)e^{-(m-n)\rho}.
\end{aligned}
\end{equation}

The estimates \eqref{eq:ac-coefficient-bound} and \eqref{eq:ac-negative-branch-bound} were proved for every $B$. Summing the Taylor remainder exactly as in the proof of Lemma~\ref{lem:ac-kernel-bounds}, now starting at $m+1$, gives, for $0<u\le e^{-1}$,
\[
 |R_{m+1,h}(u)|+|uR_{m+1,h}'(u)|\le C_Bu^{m+1},
 \qquad |F_{-,h}(u)|+|uF_{-,h}'(u)|\le C_B.
\]
Using $\partial_\rho u=-u$ and the same exponential-prefactor estimates as in that proof, the first two terms of \eqref{eq:threshold remainder decomposition}, together with their $\rho$-derivatives, are bounded respectively by
\[
 C_Be^{-r/2}e^{-(\rho-r)},\qquad
 C_Be^{-r/2}e^{-(m+1)(\rho-r)}.
\]
For the finite sum, the argument there based on \eqref{eq:landau-radial-bound} applies with $s_n-\frac12=m-n\ge1$ and gives
\[
 \sum_{j=0}^1\left|\partial_\rho^j
 \sum_{n=0}^{m-1}P^0_{s_n,B}(r)e^{-(m-n)\rho}\right|
 \le C_Be^{-r/2}\sum_{n=0}^{m-1}e^{-(m-n)(\rho+r)}
 \le C_Be^{-r/2}e^{-(\rho-r)}.
\]
When $m=0$ this sum is empty. Combining these bounds proves \eqref{eq:threshold radial tail}.
\end{proof}

\begin{theorem}
	\label{thm:threshold decay}
	Let $|B|=m+\frac12$, where $m\in\mathbb N_0$, and use the threshold residue $A_{1/2}$ from \eqref{eq:threshold residue}. The kernel $K_B^{ac}(t;z,z')$ of $e^{-itH_B}P_{ac}$ has the decomposition
	\begin{equation}\label{eq:threshold kernel splitting}
		K_B^{ac}(t;z,z')
		=\frac{e^{-itE_\ast(B)}e^{-\frac{\pi i}{4}\operatorname{sgn}(t)}}
		{\sqrt{\pi|t|}}A_{1/2}(z,z')
		+\mathcal E_B(t;z,z'),\qquad t\ne0,
	\end{equation}
	where $r=d(z,z')$ and
	\begin{equation}\label{eq:threshold remainder bound}
		|\mathcal E_B(t;z,z')|
		\le C_B|t|^{-3/2}(1+r)^2e^{-r/2}.
	\end{equation}
	For $0<|t|\le1$, the full kernel satisfies the sharper short-time bound
	\begin{equation}\label{eq:threshold short kernel bound}
		|K_B^{ac}(t;z,z')|
		\le C_B|t|^{-1}(1+r)^{1/2}e^{-r/2}.
	\end{equation}
\end{theorem}

\begin{proof}
Let $V_B$ be given by Lemma~\ref{lem:threshold-kernel-bounds}. For $\operatorname{Re}\tau>0$,
\[
 \int_0^\infty\rho e^{-\rho^2/(4\tau)}\,d\rho
 =\left[-2\tau e^{-\rho^2/(4\tau)}\right]_0^\infty=2\tau.
\]
Using $W_B^{ac}=P^0_{1/2,B}+V_B$ in \eqref{eq:ac heat kernel}, and then using \eqref{eq:threshold residue}, gives
\begin{equation}\label{eq:threshold complex heat splitting}
\begin{aligned}
 K^{H,ac}_\tau(z,z')
 ={}&\frac{e^{-\tau E_\ast(B)}}{\sqrt\pi\,\tau^{1/2}}A_{1/2}(z,z')+\frac{\Phi_B(z,z')e^{-\tau E_\ast(B)}}{8\pi^{3/2}\tau^{3/2}}
 \int_0^\infty\rho e^{-\rho^2/(4\tau)}V_B(\rho,r)\,d\rho.
\end{aligned}
\end{equation}
The powers of $\tau$ are taken on the branch positive for $\tau>0$.

\smallskip
We use the same splitting at $r$ and $r+1$ as in the proof of \eqref{eq:ac-moment-bound} in Theorem~\ref{thm:Main}. On $[0,r)$, $V_B=-S_B-P^0_{1/2,B}$, so \eqref{eq:correction-global-bound} and \eqref{eq:threshold profile bound} give
\[
 \int_0^r\rho|V_B(\rho,r)|\,d\rho\le C_Br^2e^{-r/2}.
\]
On $[r,r+1]$, the calculation there using \eqref{eq:ac-near-bound} and \eqref{eq:correction-global-bound} is unchanged, except for the additional profile, whose contribution is at most $C_B(1+r)e^{-r/2}$. On $[r+1,\infty)$, replace \eqref{eq:ac-tail-bound} by \eqref{eq:threshold radial tail} in the same calculation. Hence
\[
 \int_r^{r+1}\rho|V_B(\rho,r)|\,d\rho
 +\int_{r+1}^\infty\rho|V_B(\rho,r)|\,d\rho
 \le C_B(1+r)e^{-r/2}.
\]
Together these estimates yield
\begin{equation}\label{eq:threshold radial moment}
 \int_0^\infty\rho|V_B(\rho,r)|\,d\rho
 \le C_B(1+r)^2e^{-r/2}.
\end{equation}
As in the boundary-value argument of Theorem~\ref{thm:Main}, this moment bound permits dominated convergence with $\tau=\varepsilon+it$, since $|e^{-\rho^2/(4\tau)}|\le1$. Taking the branches positive on $\tau>0$, the second term of \eqref{eq:threshold complex heat splitting} therefore has boundary value
\begin{equation}\label{eq:threshold error integral}
 \mathcal E_B(t;z,z')
 =\frac{\Phi_B(z,z')e^{-itE_\ast(B)}e^{-3\pi i\operatorname{sgn}(t)/4}}
 {8\pi^{3/2}|t|^{3/2}}
 \int_0^\infty\rho e^{i\rho^2/(4t)}V_B(\rho,r)\,d\rho,
\end{equation}
with an absolutely convergent integral. Since $|\Phi_B|=1$, \eqref{eq:threshold radial moment} immediately gives \eqref{eq:threshold remainder bound}.

\smallskip
	It remains to prove \eqref{eq:threshold short kernel bound}. By time-reversal symmetry, it suffices to consider $0<t\le1$. Write
	\[
	M_t(r):=\int_0^\infty\rho e^{i\rho^2/(4t)}V_B(\rho,r)\,d\rho,
	\qquad A=\sqrt{r^2+t}\le r+1.
	\]
	Decompose this absolutely convergent integral as
	\begin{align*}
		M_t(r)={}&
		\int_r^A\rho e^{i\rho^2/(4t)}W_B(\rho,r)\,d\rho
		+\int_A^{r+1}\rho e^{i\rho^2/(4t)}W_B(\rho,r)\,d\rho\\
		&-\int_0^{r+1}\rho e^{i\rho^2/(4t)}S_B(\rho,r)\,d\rho
		-P^0_{1/2,B}(r)\int_0^{r+1}\rho e^{i\rho^2/(4t)}\,d\rho\\
		&+\int_{r+1}^\infty\rho e^{i\rho^2/(4t)}V_B(\rho,r)\,d\rho\\
		={}&I_0+I_1-J_S-J_0+T.
	\end{align*}
The terms $I_0,I_1$ are exactly those in the proof of Proposition~\ref{pro:upper bound}: the direct endpoint estimate for $I_0$ and the integration by parts for $I_1$, using \eqref{eq:ac-near-bound}--\eqref{eq:ac-near-derivative}, give
\[
 |I_0|+|I_1|\le C_B\sqrt t\,(1+r)^{1/2}e^{-r/2}.
\]
The same integration-by-parts estimate for $J$ in that proof gives
\[
 |J_S|\le C_Bte^{-r/2},
\]
because $s_{N-1}-\frac12=1$ when $m\ge1$; when $m=0$, $J_S=0$.
	The constant term can be integrated exactly:
	\[
	J_0=-2it\,P^0_{1/2,B}(r)\left(e^{i(r+1)^2/(4t)}-1\right),
	\qquad |J_0|\le4t|P^0_{1/2,B}(r)|\le C_Bt e^{-r/2}.
	\]
Finally, the estimate for $T$ in the proof of Proposition~\ref{pro:upper bound} applies with $V_B$ in place of $W_B^{ac}$ and \eqref{eq:threshold radial tail} in place of \eqref{eq:ac-tail-bound}. The boundary term at infinity again vanishes, and
\[
 |T|\le C_Bte^{-r/2}.
\]
	Since $t\le\sqrt t\,(1+r)^{1/2}$ in this range, we conclude that
	\[
	|M_t(r)|\le C_B\sqrt t\,(1+r)^{1/2}e^{-r/2}.
	\]
	By \eqref{eq:threshold profile bound}, \eqref{eq:threshold kernel splitting}, and \eqref{eq:threshold error integral},
	\[
	\begin{aligned}
	|K_B^{ac}(t;z,z')|
	&\le C_Bt^{-1/2}e^{-r/2}+C_Bt^{-3/2}|M_t(r)|\\
	&\le C_Be^{-r/2}\left[t^{-1/2}+t^{-1}(1+r)^{1/2}\right]\\
	&\le C_Bt^{-1}(1+r)^{1/2}e^{-r/2},\qquad 0<t\le1.
	\end{aligned}
	\]
	This proves \eqref{eq:threshold short kernel bound}.
\end{proof}

\section{Dispersive and Strichartz estimates}
\label{sec:proof of cor}
We now deduce Theorem~\ref{cor:decay} and Corollary~\ref{cor:Strichartz} from the kernel estimates in Section~\ref{sec:heat-propagator}. We first prove the $L^1\to L^\infty$ bounds for all magnetic fields, then obtain the stronger nonthreshold estimates needed for the $TT^*$ argument.

\begin{proof}[Proof of Theorem~\ref{cor:decay}]
For $B\notin\frac12+\mathbb Z$, Theorem~\ref{thm:Main} and Proposition~\ref{pro:upper bound} imply
\[
\begin{aligned}
 |K_B^{ac}(t;z,z')|
 &\le C_B|t|^{-3/2}|L_t(r)|\\
 &\le C_B\min\!\left\{|t|^{-3/2}(1+r),\,
 |t|^{-1}(1+r)^{1/2}\right\}e^{-r/2}.
\end{aligned}
\]
Since the two spatial factors are bounded, taking the supremum over $z,z'$ gives
\[
 \|e^{-itH_B}P_{ac}\|_{L^1\to L^\infty}
 \le C_B
 \begin{cases}
 |t|^{-1},&0<|t|\le1,\\
 |t|^{-3/2},&|t|\ge1.
 \end{cases}
\]
For $B\in\frac12+\mathbb Z$, Theorem~\ref{thm:threshold decay} and \eqref{eq:threshold profile bound} give, when $|t|\ge1$,
\[
\begin{aligned}
 |K_B^{ac}(t;z,z')|
 &\le C_B\left[|t|^{-1/2}e^{-r/2}
 +|t|^{-3/2}(1+r)^2e^{-r/2}\right],\\
 \|e^{-itH_B}P_{ac}\|_{L^1\to L^\infty}
 &\le C_B\left(|t|^{-1/2}+|t|^{-3/2}\right)
 \le C_B|t|^{-1/2}.
\end{aligned}
\]
Together with the short-time kernel bound \eqref{eq:threshold short kernel bound}, this yields
\begin{equation}\label{eq:threshold dispersive bound}
	\|e^{-itH_B}P_{ac}\|_{L^1(\mathbb H)\to L^\infty(\mathbb H)}
	\le C_B
	\begin{cases}
		|t|^{-1},&0<|t|\le1,\\
		|t|^{-1/2},&|t|\ge1.
	\end{cases}
\end{equation}
The two cases prove \eqref{dis:HB}.
\end{proof}

For nonthreshold fields, we follow the argument of Anker--Pierfelice~\cite{anker2009} to obtain the stronger dispersive estimate used in the Strichartz proof.

	\begin{corollary}[Dispersive estimate]
		\label{cor:decay stronger}
		Let $B\notin\frac12+\mathbb Z$ and \(2<q,\tilde{q}\leqslant \infty\). The continuous part $u(t)=e^{-itH_B}P_{ac}f$ satisfies the dispersive estimate
		\begin{equation}
			\label{eq:stronger dispersive}
			\|u(t)\|_{L^q}\le C_B\|f\|_{L^{\tilde{q}'}}\times\begin{cases}
				|t|^{-2\max\{\frac{1}{2}-\frac{1}{q},\frac{1}{2}-\frac{1}{\tilde{q}}\}},&\text{if }\quad 0<|t|<1,\\[4pt]
				|t|^{-\frac{3}{2}},&\text{if }\quad |t|\geqslant 1,
			\end{cases}
		\end{equation}
		where the constant \(C_B\) is independent of \(t\) and \(f\).
	\end{corollary}
	
	\begin{proof}
The representation in Theorem~\ref{thm:Main} and $|\Phi_B|=1$ give
\[
 \bigl|(e^{-itH_B}P_{ac}f)(z)\bigr|
 \le C|t|^{-3/2}\int_{\mathbb H}|L_t(d(z,z'))|\,|f(z')|\,d\omega(z')
 =C|t|^{-3/2}(|f|*|L_t|)(z).
\]
		Young's inequality, the Kunze--Stein estimate \eqref{eq:the Kunze–Stein phenomenon2}, and interpolation now yield \eqref{eq:stronger dispersive}, using the Lorentz bound \eqref{eq:Lorentz norm estimate}. The spatial $L^\infty$ endpoint follows directly from \eqref{eq:pointwise estimate}. More precisely, the small-time bounds follow from
		\[
		\begin{cases}
			\big\| e^{-it H_B}P_{ac} \big\|_{L^1 \to L^q} \lesssim |t|^{-3/2} \| L_t \|_{L^q} \leqslant C_{B,q} |t|^{-1}, & \forall q > 2, \\[4pt]
			\big\| e^{-it H_B}P_{ac} \big\|_{L^{q'} \to L^\infty} \lesssim |t|^{-3/2} \| L_t \|_{L^q} \leqslant C_{B,q} |t|^{-1}, & \forall q > 2, \\[4pt]
			\big\| e^{-it H_B}P_{ac} \big\|_{L^2 \to L^2} \leqslant 1,
		\end{cases}
		\]
		and the large‑time bounds follow from
		\[
		\begin{cases}
			\big\| e^{-it H_B}P_{ac} \big\|_{L^1 \to L^q} \lesssim |t|^{-3/2} \| L_t \|_{L^q} \leqslant C_{B,q} |t|^{-\frac{3}{2}}, & \forall q > 2, \\[4pt]
			\big\| e^{-it H_B}P_{ac} \big\|_{L^{q'} \to L^\infty} \lesssim |t|^{-3/2} \| L_t \|_{L^q} \leqslant C_{B,q} |t|^{-\frac{3}{2}}, & \forall q > 2, \\[4pt]
			\big\| e^{-it H_B}P_{ac} \big\|_{L^{q'} \to L^q} \leqslant C_{B,q} |t|^{-3/2} \| L_t \|_{L^{q,1}} \leqslant C_{B,q} |t|^{-\frac{3}{2}}, & \forall q > 2.
		\end{cases}
		\]
		The pointwise kernel bound and the $L^2$ estimate first give, by interpolation,
		\[
		 \|e^{-itH_B}P_{ac}\|_{L^{q'}\to L^q}
		 \le C_{B,q}|t|^{-2(\frac12-\frac1q)},
		 \qquad 0<|t|<1.
		\]
		If $\tilde q\ge q$, interpolate this diagonal estimate with the $L^1\to L^q$ estimate above; if $q\ge\tilde q$, use the dual argument. The resulting exponent is
		\[
		 2\max\left\{\frac12-\frac1q,\frac12-\frac1{\tilde q}\right\}.
		\]
		The same two-case interpolation applied to the large-time estimates preserves the common factor $|t|^{-3/2}$.
	\end{proof}

\begin{proof}[Proof of Corollary~\ref{cor:Strichartz}]
Assume $B\notin\frac12+\mathbb Z$ and set $U_B(t)=e^{-itH_B}P_{ac}$. Since $P_{ac}$ is an orthogonal projection commuting with $H_B$,
	\[
	\|U_B(t)\|_{L^2\to L^2}\leqslant1,
	\qquad U_B(t)U_B(s)^*=U_B(t-s).
	\]
	Apply the $TT^*$ argument of Keel--Tao~\cite{keel-tao} in the form used for hyperbolic space by Anker--Pierfelice~\cite{anker2009}. For $|t-s|<1$, the exponent in \eqref{eq:stronger dispersive} gives precisely the condition $2/p+2/q\ge1$; for $|t-s|\ge1$, the integrable factor $|t-s|^{-3/2}$ imposes no additional restriction. The endpoint $(p,q)=(\infty,2)$ follows from the $L^2$ bound. The same $TT^*$ argument, including its retarded form, yields the inhomogeneous estimate for every two pairs in $\Delta$.
\end{proof}


\begin{thebibliography}{Com86}

\bibitem{Agmon75}
		S.~Agmon,
		\newblock Spectral properties of {Schr\"odinger} operators and scattering theory.
		\newblock {\em Ann. Scuola Norm. Sup. Pisa Cl. Sci. (4)}, 2(2):151--218, 1975.

        
		\bibitem{AMS01}
		L.~Alili, H.~Matsumoto and T.~Shiraishi,
		\newblock On a triplet of exponential Brownian functionals.
		\newblock In {\em S\'eminaire de Probabilit\'es XXXV}, Lecture Notes in Mathematics, vol.~1755, pages 396--415. Springer, Berlin, 2001.

        
		\bibitem{anker2009}
		J. P. Anker and V. Pierfelice,
		\newblock Nonlinear {Schr{\"o}dinger} equation on real hyperbolic spaces.
		\newblock {\em Ann. Inst. Henri Poincar{\'e}, Anal. Non Lin{\'e}aire}, 26(5):1853--1869, 2009.
		
		\bibitem{Banica2007}
		V.~Banica,
		\newblock The nonlinear {Schr{\"o}dinger} equation on hyperbolic space.
		\newblock {\em Commun. Partial Differ. Equations}, 32(10):1643--1677, 2007.

        
        \bibitem{BVS13}
		J.~A. Barcel\'o, L.~Vega, and M.~Zubeldia,
		\newblock The forward problem for the electromagnetic {H}elmholtz equation with critical singularities.
		\newblock {\em Adv. Math.}, 240:636--671, 2013.

		
		\bibitem{CHMM98} A. L. Carey, K. C. Hannabuss, V. Mathai and P. McCann, Quantum Hall effect on the hyperbolic plane, Comm. Math. Phys. 190 (1998), 629-673.
		
		\bibitem{Comtet1987}
		A. Comtet,
		\newblock On the {Landau} levels on the hyperbolic plane.
		\newblock {\em Ann. Phys.}, 173(1):185--209, 1987.
		
		\bibitem{cowling1997}
		M. Cowling,
		\newblock Herz's ``{Principe} de majoration'' and the {Kunze}-{Stein} phenomenon.
		\newblock In {\em Harmonic analysis and number theory. Papers in honour of Carl S. Herz. Proceedings of the conference, April 15--19, 1996, Montr\'eal, Canada}, pages 73--88. Providence, RI: American Mathematical Society, 1997. 

        

		\bibitem{ErdoganSchlag04}
		M.~B. Erdo\u{g}an and W.~Schlag,
		\newblock Dispersive estimates for {S}chr\"odinger operators in the presence of a resonance and/or an eigenvalue at zero energy in dimension three: {I}.
		\newblock {\em Dyn. Partial Differ. Equ.}, 1(4):359--379, 2004.
		
		\bibitem{Fay77}
		 J. D. Fay, Fourier coefficients of the resolvent for a Fuchsian group, J. Reine Angew. Math. 293-294 (1977), 143-203.

		\bibitem{Folland1989}
		G.~B. Folland,
		\newblock \emph{Harmonic Analysis in Phase Space}.
		\newblock Annals of Mathematics Studies 122, Princeton University Press, Princeton, NJ, 1989.



		\bibitem{GoldbergSchlag04}
		M.~Goldberg and W.~Schlag,
		\newblock Dispersive estimates for {S}chr\"odinger operators in dimensions one and three.
		\newblock {\em Comm. Math. Phys.}, 251(1):157--178, 2004.

		\bibitem{IS72}
		T.~Ikebe and Y.~Saito,
		\newblock Limiting absorption method and absolute continuity for the {S}chr\"odinger operator.
		\newblock {\em J. Math. Kyoto Univ.}, 12:513--542, 1972.

		\bibitem{ionescu2000}
		A. D. Ionescu,
		\newblock An endpoint estimate for the {Kunze}-{Stein} phenomenon and related maximal operators.
		\newblock {\em Ann. Math. (2)}, 152(1):259--275, 2000.


        \bibitem{JensenKato79}
		A.~Jensen and T.~Kato,
		\newblock Spectral properties of {S}chr\"odinger operators and time-decay of the wave functions.
		\newblock {\em Duke Math. J.}, 46(3):583--611, 1979.
        
        \bibitem{JMP84}
		A.~Jensen, E.~Mourre, and P.~Perry,
		\newblock Multiple commutator estimates and resolvent smoothness in quantum scattering theory.
		\newblock {\em Ann. Inst. H. Poincar\'e Sect. A (N.S.)}, 41(2):207--225, 1984.
		
		\bibitem{keel-tao}
		M. Keel and T. Tao,
		\newblock Endpoint {Strichartz} estimates.
		\newblock {\em Am. J. Math.}, 120(5):955--980, 1998.

        
		\bibitem{KlausSimon80}
		M.~Klaus and B.~Simon,
		\newblock Coupling constant thresholds in nonrelativistic quantum mechanics. {I}. {S}hort-range two-body case.
		\newblock {\em Ann. Physics}, 130(2):251--281, 1980.
		
		\bibitem{Lan30} L. Landau, Diamagnetismus der Metalle, Z. Physik 64(1930), 629-637.


		\bibitem{LudewigThiang21}
		M.~Ludewig and G.~C. Thiang,
		\newblock Gaplessness of {L}andau {H}amiltonians on hyperbolic half-planes via coarse geometry.
		\newblock {\em Comm. Math. Phys.}, 386(1):87--106, 2021.
        

		\bibitem{Newton82}
		R.~G. Newton,
		\newblock {\em Scattering Theory of Waves and Particles}.
		\newblock Second edition, Springer, New York, 1982.

		
		\bibitem{Osh90} K. Oshima, Completeness relations for Maass Laplacians and heat kernels on the super Poincar\'{e} upper half-plane, J. Math. Phys. 31 (1990), 3060-3063.

		\bibitem{patterson1975}
		S.~J. Patterson,
		\newblock The {Laplacian} operator on a {Riemann} surface.
		\newblock {\em Compos. Math.}, 31:83--107, 1975.

 \bibitem{ReedSimonIV}
		M.~Reed and B.~Simon,
		\newblock {\em Methods of Modern Mathematical Physics. {IV}: Analysis of Operators}.
		\newblock Academic Press, New York, 1978.



		\bibitem{Saito79}
		Y.~Saito,
		\newblock {\em Spectral Representations for Schr\"odinger Operators with Long-Range Potentials}.
		\newblock Lecture Notes in Mathematics 727, Springer, Berlin, 1979.

		\bibitem{Saito87}
		Y.~Saito,
		\newblock {S}chr\"odinger operators with a nonspherical radiation condition.
		\newblock {\em Pacific J. Math.}, 126(2):331--359, 1987.

        


        \bibitem{Schlag21}
        W.~Schlag,
        \newblock On pointwise decay of waves.
        \newblock {\em J. Math. Phys.}, 62(6):061509, 2021.
        \newblock \url{https://doi.org/10.1063/5.0042767}.


        
		\bibitem{Shubin2001}
		M. Shubin,
		\newblock Essential self-adjointness for semi-bounded magnetic {Schr{\"o}dinger} operators on non-compact manifolds.
		\newblock {\em J. Funct. Anal.}, 186(1):92--116, 2001.

        
	\bibitem{Sommerfeld12}
		A.~Sommerfeld,
		\newblock Die {G}reensche {F}unktion der {S}chwingungsgleichung.
		\newblock {\em Jahresber. Dtsch. Math.-Ver.}, 21:309--353, 1912.

		\bibitem{Thangavelu1993}
		S.~Thangavelu,
		\newblock \emph{Lectures on Hermite and Laguerre Expansions}.
		\newblock Mathematical Notes 42, Princeton University Press, Princeton, NJ, 1993.
		


		\bibitem{SS20}
		Z.~\v{S}abanac and L.~\v{S}\'ceta,
		\newblock The heat kernel of a weighted Maass-Laplacian with real weights.
		\newblock {\em Albanian J. Math.}, 14(1):25--35, 2020.


		\bibitem{DLMF}
		NIST Digital Library of Mathematical Functions.
		\newblock Chapters~5, 15 and~18: gamma and hypergeometric functions and orthogonal polynomials.
		\newblock \url{https://dlmf.nist.gov/}.

\end{thebibliography}

\end{document}